\documentclass[reqno]{amsart}
\usepackage{amsmath}
\usepackage{amsthm}
\usepackage{amsfonts}
\usepackage{amssymb}

\usepackage{verbatim}
\usepackage{subfigure}

\newtheorem{theorem}{Theorem}[section]
\newtheorem{proposition}[theorem]{Proposition}
\newtheorem{lemma}[theorem]{Lemma}
\newtheorem{corollary}[theorem]{Corollary}

\theoremstyle{definition}
\numberwithin{equation}{section}  

\usepackage{graphicx}

\usepackage[all]{xy}
\numberwithin{equation}{section}
\def\p{\partial}

\title[Free-boundary 3D compressible Euler equations  in vacuum]
{A priori estimates for the free-boundary 
 3D compressible Euler equations in physical vacuum}

\author[D. Coutand]{Daniel Coutand}
\author[H. Lindblad]{Hans Lindblad} 
\author[S. Shkoller]{Steve Shkoller}
\address{CANPDE, Maxwell Institute for Mathematical Sciences and department of Mathematics, Heriot-Watt University, Edinburgh, EH14 4AS, UK}
\address{Department of Mathematics, University of California, San Diego, CA 92093}
\address{Department of Mathematics, University of California, Davis, CA 95616, USA}


\email{D.Coutand@ma.hw.ac.uk}
\email{lindblad@math.ucsd.edu}
\email{shkoller@math.ucdavis.edu}

\begin{document}

\begin{abstract}
We prove a priori estimates for the three-dimensional compressible Euler equations with moving {\it physical} vacuum boundary,
with an equation of state given by $p(\rho) = C_\gamma \rho^\gamma $ for $\gamma >1$.
The vacuum condition necessitates the vanishing of the pressure, and hence density, on the dynamic  boundary, which creates
a degenerate and characteristic hyperbolic {\it free-boundary} system to which standard methods of symmetrizable hyperbolic equations cannot be applied. 
\end{abstract}

\maketitle
{\small 
\tableofcontents}

\section{Introduction}
\label{sec_introduction}

\subsection{The compressible Euler equations in Eulerian variables}
For $0 \le t \le T$,
the evolution of a three-dimensional  compressible gas  moving inside of a dynamic vacuum  boundary is modeled by the one-phase compressible Euler equations:
\begin{subequations}
  \label{ceuler}
\begin{alignat}{2}
\rho[u_t+ u\cdot Du] + D p(\rho)&=0  &&\text{in} \ \ \Omega(t) \,, \label{ceuler.a}\\
 \rho_t+ {\operatorname{div}}  (\rho u) &=0 
&&\text{in} \ \ \Omega(t) \,, \label{ceuler.b}\\
p &= 0 \ \ &&\text{on} \ \ \Gamma(t) \,, \label{ceuler.c}\\
\mathcal{V} (\Gamma(t))& = u \cdot n &&\ \ \label{ceuler.d}\\
(\rho,u)   &= (\rho_0,u_0 ) \ \  &&\text{on} \ \ \Omega(0) \,, \label{ceuler.e}\\
   \Omega(0) &= \Omega\,.  && \label{ceuler.f}
\end{alignat}
\end{subequations}
The open, bounded subset
 $\Omega(t) \subset \mathbb{R}^3  $ denotes the changing volume occupied by the gas,  $\Gamma(t):= \partial\Omega(t)$ denotes
 the moving vacuum boundary, $ \mathcal{V} (\Gamma(t))$ denotes normal
 velocity of $\Gamma(t)$, and $n$ denotes the exterior unit normal vector to $\Gamma(t)$.
  The vector-field $u = (u_1,u_2,u_3)$ denotes the Eulerian velocity
field, $p$ denotes the pressure function, and $\rho$ denotes the density of the gas.
The equation of state $p(\rho)$ is given by
\begin{equation}\label{eos}
p(x,t)= C_\gamma\, \rho(x,t)^\gamma\ \  \text{ for } \ \  \gamma> 1 ,
\end{equation} 
where $C_\gamma $ is the adiabatic constant which we set to unity, and
$$
\rho>0 \ \text{ in } \ \Omega(t) \ \ \ \text{ and } \ \ \ \rho=0 \ \text{ on } \Gamma(t) \,.
$$
Equation (\ref{ceuler.a}) is the conservation of momentum; (\ref{ceuler.b}) is the conservation of mass; the boundary
condition (\ref{ceuler.c}) states that pressure (and hence density)  vanish along the vacuum boundary; (\ref{ceuler.d}) states that the
vacuum boundary is moving with the normal component of the fluid velocity, and (\ref{ceuler.e})-(\ref{ceuler.f}) are the 
initial conditions for the density, velocity, and domain.  Using the equation of state (\ref{eos}),
 (\ref{ceuler.a}) is written as
\begin{alignat}{2}
\rho [u_t+ u\cdot Du]+  D\rho^{\gamma} &=0 \ \ \   &&\text{in} \ \ \Omega(t) \,. \tag{ \ref{ceuler.a}'} 
\end{alignat}

\subsection{Physical vacuum}  \label{subsec_physicalvacuum}
With the sound speed given by $c := \sqrt{\p p/ \p \rho}$ and $N$ denoting the {\it outward} unit normal to $\Gamma$, satisfaction of the condition
\begin{equation}\label{phys_vac}
\frac{ \partial c_0^2}{ \partial N} < 0 \text{ on } \Gamma
\end{equation} 
defines a {\it physical vacuum} boundary (see \cite{Lin1987}, \cite{Liu1996}, \cite{LiYa1997}, \cite{LiYa2000},
\cite{LiSm1980}, \cite{XuYa2005}), where $c_0 = c|_{t=0}$.  The physical vacuum condition (\ref{phys_vac}) is equivalent to the
requirement that
\begin{equation}\label{stab1}
\frac{ \p \rho_0^{\gamma-1}} { \partial N} < 0 \text{ on } \Gamma \,.
\end{equation}
Since $ \rho_0 >0$ in $\Omega$, (\ref{stab1}) implies that for some positive constant $C$ and $x\in \Omega$ near the vacuum boundary $\Gamma$,
\begin{equation}\label{degen}
\rho_0^{\gamma-1}(x) \ge  C \text{dist}(x, \Gamma) \,.
\end{equation} 

Because of condition (\ref{degen}), the compressible Euler system (\ref{ceuler}) is a {\it degenerate} and  {\it characteristic} hyperbolic
system to which standard methods of symmetric hyperbolic conservation laws cannot be applied.

We note that by
choosing a lower-bound with a faster rate of degeneracy such as, for example,  $ \operatorname{dist}(x, \Gamma(t))^b$ 
for $b=2,3,....$, the analysis becomes significantly easier; for instance, if  $b=2$,
then $\frac{D\rho_0^{\gamma-1}(x,t)}{\sqrt{ \rho_0^{\gamma-1}(x,t)}}$ is bounded for all $x \in \Omega$. This bound makes it possible to
easily control error terms in energy estimates, and in effect removes the singular behavior associated with the physical vacuum condition
(\ref{degen}).

\subsection{Fixing the domain and the Lagrangian variables on $\Omega$}   We transform the system (\ref{ceuler}) into
Lagrangian variables.
We let $\eta(x,t)$ denote the ``position'' of the gas particle $x$ at time $t$.  Thus,
\begin{equation}
\nonumber
\begin{array}{c}
\partial_t \eta = u \circ \eta $ for $ t>0 $ and $
\eta(x,0)=x
\end{array}
\end{equation}
where $\circ $ denotes composition so that
$[u \circ \eta] (x,t):= u(\eta(x,t),t)\,.$  
We set 
\begin{align*}
v &= u \circ \eta   \text{ (Lagrangian velocity)},  \\
f&= \rho \circ \eta   \text{ (Lagrangian density)}, \\
A &= [D \eta]^{-1}  \text{ (inverse of deformation tensor)}, \\
J &= \det D \eta  \text{ (Jacobian determinant)}, \\
a &= J\, A  \text{ (tranpose of cofactor matrix)}. 
\end{align*}
Using  Einstein's summation convention defined in Section \ref{subsec_einstein} below, and using
the notation $F,_k$ to denote $\frac{\p F}{ \p x_k}$, the k$th$-partial derivative of $F$ for $k=1,2,3$,
the Lagrangian version of equations (\ref{ceuler.a})-(\ref{ceuler.b}) can be written on
the fixed reference domain $\Omega$ as
\begin{subequations}
\label{ceuler00}
\begin{alignat}{2}
f  v^i_t + A^k_i f^ \gamma ,_k &=0 \ \ && \text{ in } \Omega \times (0,T] \,, \label{ceuler00.a} \\
f _t + f A^j_i v^i,_j  &=0 \ \ && \text{ in } \Omega \times (0,T] \,,\label{ceuler00.b}  \\
f  &=0 \ \ && \text{ in } \Omega \times (0,T] \,,\label{ceuler00.c}  \\
(f,v,\eta)  &=(\rho_0, u_0, e) \ \  \ \ && \text{ in } \Omega \times \{t=0\} \,, \label{ceuler00.d} 
\end{alignat}
\end{subequations}
where $e(x)=x$ denotes the identity map on $\Omega$.

Since $J_t = J A^j_i v^i,_j$ and since $J(0)=1$ (since we have taken $\eta(x,0)=x$), it follows that
\begin{equation}\label{J}
f = \rho_0 J^{-1},
\end{equation}
so that the initial density function $\rho_0$ can be viewed as a parameter in the Euler equations.   Let $\Gamma:= \partial \Omega$ denote
the initial vacuum boundary;  using, that $A^k_i = J ^{-1} \, a^k_i$, we write the compressible Euler equations (\ref{ceuler00}) as
\begin{subequations}
\label{ceuler0}
\begin{alignat}{2}
\rho_0 v^i_t + a^k_i (\rho_0^ \gamma J^{-\gamma}),_k &=0 \ \ && \text{ in } \Omega \times (0,T] \,, \label{ceuler0.a} \\
(\eta, v)  &=( e, u_0) \ \  \ \ && \text{ in } \Omega \times \{t=0\} \,, \label{ceuler0.b} \\
\rho_0^{\gamma-1}& = 0 \ \ &&\text{ on }  \Gamma \,, \label{ceuler0.c}  
\end{alignat}
\end{subequations}
with $ \rho _0^{ \gamma -1}(x) \ge C \operatorname{dist}( x, \Gamma ) $ for $x \in\Omega$ near $\Gamma$.

\subsection{Setting $\gamma=2$}  It should be clear from equations (\ref{ceuler0}) that by introducing new variables for
both $\rho_0^{\gamma-1}$ and $J^{\gamma-1}$, such as the enthalpy for example, we can always return to the case that $\gamma=2$.  Henceforth, we seek solutions $\eta$ to the following system:
\begin{subequations}
  \label{ce0}
\begin{alignat}{2}
\rho_0  v_t^i + a^k_i (\rho_0^2 J^{-2}),_k&=0  &&\text{in} \ \ \Omega \times (0,T] 
\,, \label{ce0.a}\\
(\eta,v)&= (e,u_0 ) \ \ \ &&\text{on} \ \ \Omega \times \{t=0\} \,, \label{ce0.b}\\
\rho_0& = 0 \ \ &&\text{ on }  \Gamma \,, \label{ce0.c}
\end{alignat}
\end{subequations}
with $ \rho _0(x) \ge C \operatorname{dist}( x, \Gamma ) $ for $x \in\Omega$ near $\Gamma$.

The equation
(\ref{ce0.a}) is equivalent to 
\begin{equation}
v_t^i + 2A^k_i (\rho_0 J^{-1} ),_k =0 \label{ce_vor} \,,
\end{equation} 
and (\ref{ce_vor}) can be written as
 
\begin{equation}
v_t^i +\rho_0 a^k_i J^{-2},_k + \rho_0,_3 a^3_i J^{-2}  =0 \label{ce_elliptic} \,.
\end{equation} 
Because of the degeneracy caused by $\rho_0 =0$ on $\Gamma$, all three equivalent forms of the compressible
Euler equations are crucially used in our analysis.  The equation (\ref{ce0.a}) is used for energy estimates, while (\ref{ce_vor}) is
used for estimates of the vorticity, and (\ref{ce_elliptic}) is used for additional elliptic-type estimates used to recover the bounds
for normal derivatives.

\subsection{The reference domain $\Omega$}\label{subsec_domain}
 To avoid the use of local coordinate charts necessary for arbitrary geometries,
for simplicity, we will assume  that the initial domain $\Omega \subset \mathbb{R}  ^3$ at time $t=0$ is given by
$$
\Omega = \{ (x_1,x_2,x_3) \in \mathbb{R}^3  \ | \ (x_1, x_2) \in  \mathbb{T}  ^2 , \ x_3\in (0,1) \} \,,
$$
where $\mathbb{T}  ^2$ denotes the $2$-torus, which can be thought of as the unit square with periodic
boundary conditions.  This permits the use of {\it one} global Cartesian coordinate system. At $t=0$, the reference {\it vacuum} boundary  is the
{\it top} boundary
$$\Gamma =\{ x_3=1\}\,,$$
while the {\it bottom} boundary $\{x_3=0\}$ is fixed with boundary condition
$$
u^3=0 \text{ on } \{ x_3 =0 \} \times [0,T] \,.
$$
The moving vacuum boundary is then given by
$$
\Gamma(t) = \eta(t)(\Gamma) = \eta(x_1,x_2,1,t) \,.
$$

 \subsection{The higher-order energy function} The physical energy $  \int_\Omega \bigl[{\frac{1}{2}}  \rho_0 |v|^2 +  \rho_0^2 J^{-1} \bigl]dx$ is a conserved quantity, but is far too weak for the
purposes of constructing solutions; instead, we consider the higher-order energy function
\begin{align} 
E(t) & = \sum_{a=0}^4 \| \p_t^{2a}  \eta(t)\|^2_{4-a} + \sum_{a=0}^4 \Bigl[\| \rho_0 \bar\p^{4-a}\p_t^{2a} D\eta(t)\|^2_0
+\| \sqrt{\rho_0} \bar\p^{4-a}\p_t^{2a}  v(t)\|^2_0 \Bigr] \nonumber\\
& \qquad\qquad  
+ \sum_{a=0}^3 \| \rho_0 \p_t^{2a}  J^{-2} (t)\|^2_{4-a} + \| \operatorname{curl} _\eta v(t)\|^2_ 3 
+ \| \rho_0 \bar \p^4 \operatorname{curl} _\eta v(t)\|^2_0 \,, \label{formal}
\end{align} 
where $\bar \p = \left(\frac{\p}{\p x_1}, \frac{\p}{\p x_2}\right).$   Section \ref{notation} explains the notation.

While this function is
not conserved, it is possible to show that $ \sup_{t \in [0,T]} E(t)$ remains bounded for sufficiently smooth solutions of (\ref{ce0}), whenever $T>0$ is taken sufficiently small;   the bound depends only on $E(0)$.

 \subsection{Main Result}   
\begin{theorem} \label{theorem_main}
 Suppose that $\eta(t)$ is a smooth solution of (\ref{ce0}) on a time interval $[0,\bar T]$. Then for $0< T \le T_0$ taken sufficiently small, the energy
 function $E(t)$ constructed from the solution $\eta(t)$ satisfies the {\it a priori} estimate
 $$
 \sup_{t \in [0,T]} E(t) \le M_0 \,,
 $$
where $M_0$ and $T_0$ is a function of $E(0)$.
\end{theorem}

Of course, our theorem also covers the case that $\Omega \subset  \mathbb{R}   ^d$ for $d=1$ or $2$, and by using a collection of
coordinate charts, we can allow arbitrary initial domains, as long as the initial boundary is of Sobolev class $H^{3.5}$.
We announced Theorem \ref{theorem_main} in \cite{CoLiSh2007}.

\subsection{History of prior results for the compressible Euler equations with vacuum boundary}  
We are aware of only a handful of previous theorems pertaining to the existence of solutions to the compressible and inviscid Euler equations with moving  vacuum boundary.   Makino
 \cite{M1986} considered compactly supported initial data, and treated the compressible  Euler equations for a gas as being set on  $\mathbb{R}^3  \times (0,T]$.  With
 his methodology, it is not possible to track the location of the vacuum boundary (nor is it necessary); nevertheless, an existence theory was developed in
 this context, by a variable change that permitted the standard theory of symmetric hyperbolic systems to be employed.  Unfortunately, the constraints on the data are too severe to allow for the evolution of the physical vacuum boundary.
 
 In \cite{Li2005b}, Lindblad proved existence and uniqueness for the 3D compressible Euler equations modeling a {\it liquid} rather than a gas.
 For a compressible liquid, the density $\rho>0$ is assumed to be a positive constant on the moving vacuum boundary $\Gamma(t)$ and is
 thus uniformly bounded below  by a positive constant.  As such, the compressible liquid provides a uniformly hyperbolic, but characteristic, system.  Lindblad used Lagrangian variables combined with Nash-Moser iteration to construct solutions.   More recently,  Trakhinin \cite{Tr2008}
provided an alternative proof for the existence of a compressible liquid, employing a solution strategy based on symmetric hyperbolic systems
combined with Nash-Moser iteration.  

The only existence theory for the physical vacuum singularity that we are aware of can be found in the recent paper by Jang and
Masmoudi \cite{JaMa2008} for the 1D compressible gas;  we refer the interested reader to  the introduction in that paper for a nice history of the analysis of the 1D compressible Euler equations with damping.

\subsection{Generalization of the isentropic gas assumption}  The  general form of the compressible Euler equations in three
space dimensions are the $5 \times 5$ system of conservation laws
\begin{subequations}
  \label{claw}
\begin{alignat}{2}
\rho[u_t+ u\cdot  D u] + Dp(\rho)&=0 \,,   \label{claw.a}\\
 \rho_t+ {\operatorname{div}}  (\rho u) &=0 \,, \label{claw.b} \\
  (\rho {\mathfrak E})_t+ {\operatorname{div}}  (\rho u\mathfrak{E} + p u) &=0 \,,  \label{claw.c}
\end{alignat}
\end{subequations}
where (\ref{claw.a}), (\ref{claw.b}) and (\ref{claw.c}) represent the respective conservation of  momentum, mass, and 
total energy.  Here, the quantity $\mathfrak{E}$ is the sum of contributions from the kinetic energy $ {\frac{1}{2}} |u|^2$, and
the internal energy $e$, i.e.,$\mathfrak{E}= {\frac{1}{2}} |u|^2 + e$.  For a single phase of compressible liquid or gas, $e$ becomes a
well-defined function of $\rho$ and $p$ through the theory of thermodynamics, $e = e(\rho, p)$.  Other interesting and useful
physical quantities, the temperature $T(\rho,p)$ and the entropy $S(\rho,p)$ are defined through the following consequence of the
second law of thermodynamics
$$
T\ dS = de = - \frac{p}{ \rho^2} \ d\rho  \,.
$$
For {\it ideal gases}, the quanities $e,T,S$ have the explicit formulae:
\begin{align*} 
e(\rho, p) & = \frac{p}{\rho(\gamma -1)} = \frac{T}{\gamma -1} \\
T(\rho,p) &= \frac{p}{\rho} \\
p &= e^S \rho ^ {\gamma}, \ \ \ \gamma > 1,  \ \text{ constant} \,.
\end{align*} 
In regions of smoothness, one often uses velocity and a convenient choice of two additional variables among the five
quantities $S,T,p,\rho,e$ as independent variables.  For the Lagrangian formulation, the entropy $S$ plays an important role, as it satisfies the
transport equation
\begin{equation}\nonumber
S_t + (u \cdot D )S =0 \,,
\end{equation} 
 and as such, $S \circ \eta = S_0$, where $S_0(x) = S(x,0)$ is the initial entropy function.   Thus, by replacing $f$ with $e^{S \circ \eta} \rho_0^{\gamma} J ^{-\gamma} $, our analysis for the isentropic case naturally generalizes to the $5 \times 5$ system of conservation laws.

\section{Notation and Weighted Spaces}\label{notation}

\subsection{Differentiation and norms in the open set $\Omega$} The reference domain $\Omega$ is defined in Section 
\ref{subsec_domain}.
 Throughout the paper the symbol $D$  will be used to 
 denote the three-dimensional gradient vector
 $$D=\left(\frac{\p}{\p x_1}, \frac{\p}{\p x_2},\frac{\p}{\p x_3}\right) \,.$$

For integers $k\ge 0$ and a smooth, open domain $\Omega$ of $\mathbb{R}  ^3$, 
we define the Sobolev space $H^k(\Omega)$ ($H^k(\Omega; {\mathbb R}^3 )$) to
be the completion of $C^\infty(\Omega)$ ($C^\infty(\Omega; {\mathbb R}^3)$) 
in the norm
$$\|u\|_k := \left( \sum_{|a|\le k}\int_\Omega \left|   \frac{\p^{a_1}}{\p x_{a_1}} \frac{\p^{a_2}}{\p x_{a_2}} \frac{\p^{a_3}}{\p x_{a_3}}u(x)
\right|^2 dx\right)^{1/2},$$
for a multi-index $a \in {\mathbb Z} ^3_+$, with the standard convention that  $|a|=a_1 +a_2+ a _3$. 
For real numbers $s\ge 0$, the Sobolev spaces $H^s(\Omega)$ and the norms $\| \cdot \|_s$ are defined by interpolation.
We will  write $H^s(\Omega)$ instead of $H^s(\Omega;{\mathbb R} ^3)$ 
for vector-valued functions.  In the case that $s\ge 3$, the above definition also holds for domains
 $\Omega$  of class $H^s$.

\subsection{Tangent and normal vectors to $\Gamma$}\label{normal_and_tangent} The outward-pointing unit normal vector to $\Gamma$ is given by
$$
N =(0,0,1) \,.
$$
Similarly, the unit tangent vectors on $\Gamma$ are given by
$$
T_ 1  = (1,0,0) \, \ \ \text{ and } \ \ 
T_ 2  = (0,1,0) \,.$$

\subsection{Einstein's summation convention} \label{subsec_einstein} Repeated Latin indices $i,j,k,$, etc., are summed
from $1$ to $3$, and repeated greek indices $ \alpha , \beta , \gamma $, etc., are summed from $1$ to 
$2$.   For example, $F,_{ii} := \sum_{i=1,3} \frac{\p^2}{\p x_i \p x_i}$, and $F^i,_ \alpha I^{\alpha
\beta}  G^i,_\beta :=
\sum_{i=1}^3 \sum_{\alpha=1}^2 \sum_{\beta=1}^2  \frac{ \partial F^i}{ \partial x_ \alpha } 
I^{\alpha \beta } \frac{ \partial G^i}{ \partial x_\beta} $.

\subsection{Sobolev spaces on $\Gamma$}   For  functions $u\in H^k(\Gamma)$, $k \ge 0$,  we set
$$|u|_k := \left( \sum_{|a|\le k}\int_\Omega \left|   \frac{\p^{a_1}}{\p x_{a_1}} \frac{\p^{a_2}}{\p x_{a_2}} u(x)
\right|^2 dx\right)^{1/2},$$
for a multi-index $a \in {\mathbb Z} ^2_+$.
   For real $s \ge 0$, the Hilbert space $H^s(\Gamma)$ and the boundary norm $| \cdot |_s$ is defined by interpolation.  The negative-order
Sobolev spaces $H^{-s}(\Gamma)$ are defined via duality: for  real $s \ge 0$,
$$
H^{-s}(\Gamma) := [ H^s(\Gamma)]' \,.
$$

\subsection{Notation for derivatives and norms}
Throughout the paper, we will use the following notation:
\begin{align*}
D & = \text{ three-dimensional gradient vector } = \left( \frac{ \partial }{ \partial x_1}  , \frac{ \partial }{ \partial x_2},
\frac{ \partial }{ \partial x_3} \right) \,, \\
\bar \p & = \text{ two-dimensional gradient vector or {\it horizontal derivative} } = \left( \frac{ \partial }{ \partial x_1}  , \frac{ \partial }{ \partial x_2} \right) \,, \\
\| \cdot \|_s & = \text{  $H^s(\Omega)$ interior norm}   \,, \\
| \cdot |_s & = \text{  $H^s(\Gamma)$ boundary norm}   \,.
\end{align*} 
The $k$th partial derivative of $F$ will be denoted by $F,_k = \frac{ \p F}{ \p x_k}$.

\subsection{The embedding of a weighted Sobolev space} Using $d$ to denote the distance function to
the boundary $\Gamma$,  and letting $p=1$ or $2$,
the weighted Sobolev space  $H^1_{d^p}(\Omega)$, with norm given by 
$\int_\Omega d(x)^p (|F(x)|^2+| DF (x)|^2 )\, dx$ for any $F \in H^1_{d^p}(\Omega)$, 
satisfies the following embedding:
$$H^1_{d^p}(\Omega) \hookrightarrow   H^{1 - \frac{p}{2}}(\Omega)\,,$$
so that there is a constant $C>0$ depending only on $\Omega$ and $p$ such that
\begin{equation}\label{w-embed}
\|F\|_{1-p/2} ^2 \le C \int_\Omega d(x)^p \bigl( |F(x)|^2 + \left| DF(x) \right|^2\bigr) \, dx\,.
\end{equation} 
See, for example,  Section 8.8 in Kufner \cite{K1985}.

\section{The  Lagrangian vorticity }

We make use of the permutation symbol 
$$
\varepsilon_{ijk} = \left\{\begin{array}{rl}
1, & \text{even permutation of } \{ 1, 2, 3\}, \\
-1, & \text{odd permutation of } \{ 1, 2, 3\}, \\
0, & \text{otherwise}\,,
\end{array}\right.
$$
and the basic identity regarding the $i$th component of the curl of a vector
field $u$:
$$
(\operatorname{curl}  u)_i = \varepsilon_{ijk} u^k,_j \,.
$$
The chain rule shows that
$$
(\operatorname{curl}  u(\eta))_i =  \operatorname{curl} _\eta v: = \varepsilon _{ijk} A^s_j v^k,_s \,,
$$
the right-hand side defining the Lagrangian curl operator $ \operatorname{curl} _ \eta $.  
Taking the Lagrangian curl of (\ref{ce_vor}) yields the Lagrangian
vorticity equation
\begin{equation} \label{vorticity}
\varepsilon _{kji} A^s_j  v_t^i,_s
=0\,, \ \ \text{ or } \ \ \operatorname{curl} _ \eta v_t =0 \,.
\end{equation}

\section{Properties of the determinant $J$, cofactor matrix $a$, unit normal $n$, and a polynomial-type inequality}
\subsection{Differentiating the Jacobian determinant}  The following identities will be
useful to us:
\begin{align}
{\bar\partial}J&=  a^s_r{\bar\partial} \frac{\partial \eta^r}{\partial x^s} {\text{ (horizontal differentiation )}}\,, \label{J1}\\
\partial_t  J&= a^s_r \frac{\partial v^r}{\partial x^s}  \ \ \text{ (time differentiation using $v=\eta_t$)} \,. \label{J2}
\end{align}

\subsection{Differentiating the cofactor matrix}  Using (\ref{J1}) and (\ref{J2}) and the fact that $a= J\, A$,
we find that
\begin{align}
{\bar\partial} a^k_i &=  {\bar\partial} \frac{\partial \eta^r}{\partial x^s} J^{-1}
[a^s_r a^k_i - a^s_i a^k_r]  
{\text{ (horizontal differentiation)} }\,, \label{a1}\\
\partial_t  a^k_i &=  \frac{\partial v^r}{\partial x^s}  J^{-1} 
[a^s_r a^k_i - a^s_i a^k_r]  
\ \ \text{ (time differentiation using $v=\eta_t$)} \,. \label{a2}
\end{align}

\subsection{The Piola identity}  It is a fact that the columns of every cofactor matrix
are divergence-free and satisfy
\begin{equation}\label{a3}
a^k_i, _k =0\,.
\end{equation}
The identity (\ref{a3}) will play a vital role in our energy estimates. (Note that we use
the notation cofactor for what is commonly termed the {\it adjugate matrix}, or the transpose
of the cofactor.)

\subsection{Geometric identities}
The vectors 
$\eta,_ \alpha $ for $ \alpha =1,2$ span the tangent plane to the surface $\Gamma $ in $\mathbb{R}  ^3$, and
$$
{\tau}_ 1 : =\frac{\eta,_ 1}{|\eta,_1 |} \,, \ \  {\tau}_ 2 : =\frac{\eta,_ 2}{|\eta,_2 |} \,, \ \ \text{ and }  n:= \frac{\eta,_1 \times \eta,_2}{ | \eta,_1 \times \eta,_2|}
$$
are the unit tangent and normal vectors, respectively, to $\Gamma$.    

Let ${g}_{ \alpha \beta }= \eta,_ \alpha  \cdot \eta,_ \beta $ denote the induced metric on the surface $\Gamma$; then
$\det g = | \eta,_1 \times \eta,_2|^2$  so that
$$
\sqrt{g}\, n:= \eta,_1 \times \eta,_2\,,
$$
where we will use the notation $\sqrt{g}$ to mean $\sqrt{ \det g}$.

By definition of the cofactor matrix, 
\begin{equation}\label{a3i}
{a}^3_i = \left[
\begin{array}{c}
\eta^2,_1 \eta^3,_2 - \eta^3,_1 \eta^2,_2 \\
\eta^3,_1 \eta^1,_2 - \eta^1,_1 \eta^3,_2 \\
\eta^1,_1 \eta^2,_2 - \eta^1,_2 \eta^2,_1 
\end{array}\right] \,, \text{ and } \sqrt{g} = | {a}^3_i| \,.
\end{equation}
It follows that
\begin{equation}\label{nisa3i}
n = {a}^3_i/ \sqrt{g}\,.
\end{equation} 

We will often make use of the following differentiation formulas for the unit normal and tangent vectors:
\begin{align*}
n,_ \alpha & = -g^{ \gamma \beta } (\eta,_ { \alpha \beta } \cdot n) \, \eta,_\gamma   \,, \\
n_t  &= -g^{ \gamma \beta } ({v},_ { \beta } \cdot n)\, \eta,_ \gamma    \,, 
\end{align*}
where $g^{ \gamma \beta }$ denote the inverse of the metric ${g}_{ \gamma \beta }$. Note that the right-hand sides
these identities are tangent vectors to the embedded surface.

\subsection{A polynomial-type inequality} \label{subsec_poly} For a constant $M_0\ge 0$,  suppose that $f(t)\ge 0$, 
$t \mapsto f(t)$ is continuous,  and
\begin{equation}\label{f}
f(t) \le M_0 + C\,t\, P(f(t))\,,
\end{equation}
where $P$ denotes a polynomial function,  and  $C$ is a generic constant.
Then for $t$ taken sufficiently small, we have the bound
$$
f(t) \le 2M_0\,.
$$
This type of inequality, which we introduced in  \cite{CoSh2006}, can be viewed
as  a generalization of standard nonlinear Gronwall inequalities.

With $E(t)$ defined by (\ref{formal}), we will show that $\sup_{t \in [0,T]} E(t)$ satisfies the inequality (\ref{f}).

\section{Trace estimates and the Hodge decomposition elliptic estimates}
The normal trace theorem which states that the existence of the
normal trace of a velocity field $w\in L^2(\Omega)$ relies on the
regularity of ${\operatorname{div}}  w$ (see, for example, \cite{Temam1984}). If ${\operatorname{div}}  w\in
H^1(\Omega)'$, then $w\cdot N$, the normal trace, exists in
$H^{-0.5}(\Gamma)$ so that
\begin{align}
\|w\cdot N\|^2_{H^{-0.5}(\Gamma)} \le C \Big[\|w\|^2_{L^2(\Omega)}
+ \|{\operatorname{div}}   w\|^2_{H^1(\Omega)'}\Big] \label{normaltrace}
\end{align}
for some constant $C$ independent of $w$. In addition to the normal
trace theorem, we have the following.
\begin{lemma}\label{tangential_trace}
Let $w\in L^2(\Omega)$ so that ${\operatorname{curl}}  w\in H^1(\Omega)'$, and let
$T_1$, $T_2$ denote the unit tangent vectors on $\Gamma$,
so that any vector field $u$ on $\Gamma$ can be uniquely written as $u^\alpha
T_\alpha$. Then
\begin{align}
\|w\cdot T_\alpha\|^2_{H^{-0.5}(\Gamma)} \le C
\Big[\|w\|^2_{L^2(\Omega)} + \| {\operatorname{curl}} 
w\|^2_{H^1(\Omega)'}\Big]\,,\qquad\alpha=1,2 \label{tangentialtrace}
\end{align}
for some constant $C$ independent of $w$.
\end{lemma}
See \cite{ChCoSh2007} for the proof.
Combining (\ref{normaltrace}) and (\ref{tangentialtrace}), 
\begin{align}
\|w\|_{H^{-0.5}(\Gamma)} \le C\Big[\|w\|_{L^2(\Omega)} + \|{\operatorname{div}} 
w\|_{H^1(\Omega)'} + \|{\operatorname{curl}}  w\|_{H^1(\Omega)'}\Big]
\label{tracetemp}
\end{align}
for some constant $C$ independent of $w$.

The construction of our higher-order energy function is based on the following Hodge-type elliptic estimate:
\begin{proposition}\label{prop1}
For an $H^r$ domain $\Omega$, $r \ge 3$,
if $F \in L^2(\Omega;{\mathbb R} ^3)$ with $\operatorname{curl}F \in H^{s-1}(\Omega;{\mathbb R} ^3)$,
${\operatorname{div}}F\in H^{s-1}(\Omega)$, and $F \cdot N|_{\Gamma} \in
H^{s -{\frac{1}{2}}}(\Gamma)$ for $1 \le s \le r$, then there exists a
constant $\bar C>0$ depending only on $\Omega$ such that
\begin{equation}
\begin{array}{c}
\|F\|_s \le \bar C\left( \|F\|_0 + \|\operatorname{curl} F\|_{s-1}
+ \|\operatorname{div} F\|_{s-1} + |\bar \p F \cdot N|_{s-{\frac{3}{2}}}\right)\,, \\
\|F\|_s \le \bar C\left( \|F\|_0 + \|\operatorname{curl} F\|_{s-1}
+ \|\operatorname{div} F\|_{s-1} + |\bar \p F \cdot T_ \alpha |_{s-{\frac{3}{2}}}\right)\,,
\end{array}
\label{hodge}
\end{equation}
where $N$ denotes the outward unit-normal to $\Gamma$, and $T_ \alpha$ are tangent vectors for $ \alpha =1,2$.
\end{proposition}
The first estimate is well-known and follows from the identity $-\Delta F= {\operatorname{curl}}\, 
{\operatorname{curl}}F - D {\operatorname{div}}F$; a convenient reference is Taylor
\cite{Taylor1996}. The second estimate follows from the first using the same geometric identities on the boundary.

\section{The a priori estimates}\label{sec_usersguide}

Since the degeneracy of the initial density is only in  the normal (or vertical) direction to the vacuum boundary, 
and hence there is a constant $C>0$ such that $| \bar \p \rho_0(x)| \le C \rho_0(x)$,
 we {\it may assume without loss of generality} that
 $\rho_0 = \rho_0(x_3) \text{ and } \rho_0,_3(x_3) =  1$ for $x_3$ very small.   In fact, it
is convenient to suppose that
$$
\rho_0(x_3) = 1-x_3 \,,
$$
although any sufficiently smooth function which vanishes on $\Gamma$ and is bounded from below by a constant multiple of
the distance function near $\Gamma$ would suffice.

\subsection{Curl Estimates}\label{subsec_curlestimates}

Following Lemma 10.1 in \cite{CoSh2007}, we obtain the following estimates.

\begin{proposition} \label{curl_est}
For all $t \in (0,T)$,
\begin{align}
&\sum_{a=0}^3 \|{\operatorname{curl}} \, \p_t^{2a}\eta(t)\|_{3-a}^2 +
\sum_{l=0}^4 \|\rho_0 \, \bar \p^{4-l}  {\operatorname{curl}}\, \p_t^{2l}\eta(t)\|_{0}^2 
\le M_0 + C\, T\, P({\sup_{t\in[0,T]}} E(t))\,.
\label{curl_estimate}
\end{align}
\end{proposition}
\begin{proof} From (\ref{vorticity}), $( \operatorname{curl} _\eta v)^k_t = \varepsilon_{kji}{ A_t}^s_j v^i,_s = : B(A,Dv)$, where
$B$ is quadratic in its arguments; hence,
\begin{equation}\label{curlv_3d}
\operatorname{curl} _\eta v(t) = \operatorname{curl} u_0 + \int_0^t B(A(t'),Dv(t')) dt'\,,
\end{equation} 
and computing the gradient of this relation yields
$$
\operatorname{curl} _\eta D v(t) = \operatorname{curl} Du_0 - \varepsilon_{ \cdot ji}DA^s_j v^i,_s + \int_0^t DB(A(t'),Dv(t')) dt'\,.
$$
Applying the fundamental theorem of calculus once again, shows that
$$
\operatorname{curl} _ \eta D \eta(t) = t \operatorname{curl} Du_0 +  \varepsilon_{ \cdot ji}\int_0^t [{A_t}^s_j D\eta^i,_s -DA^s_j v^i,_s] dt'
+ \int_0^t  \int_0^{t'} D B(A(t''),Dv(t'')) dt'' dt' \,,
$$
and finally that
\begin{align} 
\operatorname{curl} D \eta(t) &= t \operatorname{curl}D u_0 - \varepsilon_{ \cdot ji}\int_0^t {A_t}^s_j(t') dt' \,  D\eta^i,_s \nonumber  \\
& +  \varepsilon_{ \cdot ji}\int_0^t [{A_t}^s_j D\eta^i,_s -DA^s_j v^i,_s] dt'
+ \int_0^t  \int_0^{t'} DB(A(t''),Dv(t'')) dt'' dt'  \,. \label{ssscurl1}
\end{align} 

To obtain an estimate for $\| \operatorname{curl} \eta(t)\|^2_3$, we let $D^2$ act on (\ref{ssscurl1}).   
With $ \p_t A^s_j = - A^s_l v^l,_p A^p_j$ and $ D A^s_j = - A^s_l D\eta^l,_p A^p_j$, we see that the first three terms on the
right-hand side of (\ref{ssscurl1}) are bounded by $M_0 + C\, T\, P( \sup_{t \in [0,T]} E(t))$, where we remind the reader that
$M_0 = P(E(0))$ is a polynomial function of the $E$ at time $t=0$.   Since
$$
DB(A, Dv) = - \varepsilon_{kji} [Dv^i,_s A^s_l v^l,_p A^p_j + v^i,_s A^s_l Dv^l,_p A^p_j+ v^i,_sv^l,_pD( A^s_l  A^p_j)],
$$
the highest-order term arising from the action of $D^2$ on $DB(A,Dv)$ is written as
$$
-\varepsilon_{kji}\int_0^t \int_0^{t'} [D^3v^i,_s A^s_l v^l,_p A^p_j + v^i,_s A^s_l D^3v^l,_p A^p_j] dt''dt'  \,.
$$
Both summands in the integrand scale like $D^3v\, Dv\, A\, A$.  The precise structure of this summand is not very important; rather,
the derivative count is the focus.  
Integrating by parts in time, 
\begin{align*} 
\int_0^t\int_0^{t'} D^3v\, Dv\, A\, A \,dt'' dt' = - \int_0^t\int_0^{t'} D^3\eta\, (Dv\, A\, A )_t dt'' dt'  + \int_0^t D^3\eta\, Dv\, A\, A\, dt'
\end{align*} 
from which it follows that
$$
\bigl\|  \int_0^t  \int_0^{t'} D^3 B(A(t''),Dv(t'')) dt'' dt'\bigr\|^2_0 \le C\, T\, P( \sup_{t \in [0,T]} E(t)) \,,
$$
and hence
$$
\sup_{t \in [0,T]} \| \operatorname{curl} \eta(t)\|^2_3 \le M_0  + C\, T\, P( \sup_{t \in [0,T]} E(t)) \,.
$$

Next, we show that 
\begin{equation} \label{sscurl2_2d}
\| \operatorname{curl} v_t(t) \|^2_2 \le M_0 + C\, T\, P({\sup_{t\in[0,T]}} E(t)) \,.
\end{equation} 
From (\ref{vorticity}),
\begin{align*} 
\operatorname{curl} v_t  &=   \varepsilon_{j\cdot i} \int_0^t {A_t}^s_j (t') dt' \, v_t^i,_s\,.
\end{align*} 
Since $H^2(\Omega)$ is a multiplicative algebra, we can directly estimate the $H^2(\Omega)$-norm of $ \operatorname{curl} v_t$ to prove
that (\ref{sscurl2_2d}) holds.   The estimates for $ \operatorname{curl} v_{ttt}(t)$ in $H^1(\Omega)$ and $ \operatorname{curl} \p_t^5v(t)$
in $L^2(\Omega)$ follow the same argument.

The weighted estimates follow from similar reasoning.  We first show that
\begin{equation}\label{sscurl13_2d}
\| \rho_0 \bar \p^4 \operatorname{curl} \eta(t) \|^2_0 \le M_0 + C\,T\, P( \sup_{t \in [0,T]} E(t)) \,.
\end{equation}
To prove this weighted estimate, we write (\ref{curlv_3d}) as
\begin{equation}\nonumber
\operatorname{curl} v(t) =  \varepsilon_{jki} v^i,_s \int_0^t {A_t}^s_j(t')dt' + \operatorname{curl} u_0 + \int_0^t B(A(t'),Dv(t')) dt'\,,
\end{equation} 
and integrate in time to find that
\begin{equation}\nonumber
\operatorname{curl} \eta(t) = t \operatorname{curl} u_0+ \int_0^t \varepsilon_{jki} v^i,_s \int_0^{t'} {A_t}^s_j(t'')dt'' dt'  + \int_0^t\int_0^{t'} B(A(t''),Dv(t'')) dt''dt'\,.
\end{equation} 

It follows that
\begin{align} 
&\rho_0 \bar \p^4 \operatorname{curl} \eta(t) =  t\rho_0 \bar \p^4 \operatorname{curl} u_0  \nonumber  \\
&\
+ \int_0^t\int_0^{t'} \varepsilon_{kji}{A_t}^s_j \rho_0 \bar \p^4 v^i,_s dt''dt'
+ \int_0^t\int_0^{t'} \varepsilon_{kji} \rho_0 \bar \p^4{A_t}^s_j  v^i,_s dt''dt'
\nonumber \\
& \ 
+ \int_0^t \varepsilon_{jki} \rho_0 \bar \p^4 v^i,_s \int_0^{t'} {A_t}^s_j(t'')dt'' dt' 
+ \int_0^t \varepsilon_{jki}  v^i,_s \int_0^{t'} \rho_0 \bar \p^4{A_t}^s_j(t'')dt'' dt' 
+ \mathfrak{R}_2  \,, \label{sscurl8_2d}
\end{align} 
where $ \mathfrak{R}_2 $ denotes terms which are  lower-order  in the derivative count; in particular  the terms with the highest
derivative count in $ \mathfrak{R}_2$ scale like $\rho \bar \p^3 Dv$ or $\rho \bar\p^4\eta$,
and hence satisfy  the inequality $ \| \mathfrak{R}_2(t)\|^2_0 \le M_0 + C\, T\, P( \sup_{t \in [0,T]} E(t))$.  We focus on 
the first integral on the right-hand side of (\ref{sscurl8_2d}); integrating by parts in time, we find that
$$
 \int_0^t\int_0^{t'} \varepsilon_{kji}{A_t}^s_j \rho_0 \bar \p^4 v^i,_s dt''dt'
= 
- \int_0^t\int_0^{t'} \varepsilon_{kji}{A_{tt}}^s_j \rho_0 \bar \p^4 \eta^i,_s dt''dt'
+ \int_0^t  \varepsilon_{kji}{A_{t}}^s_j \rho_0 \bar \p^4 \eta^i,_s   dt'  
$$
and hence
$$
\left\|   \int_0^t\int_0^{t'} \varepsilon_{kji}{A_t}^s_j \rho_0 \bar \p^4 v^i,_s dt''dt'  \right\|^2_0 \le M_0+ C\,T\, P( \sup_{t \in [0,T]} E(t)) \,.
$$
The other time integrals in (\ref{sscurl8_2d}) can be estimated in the same fashion, which proves that (\ref{sscurl13_2d}) holds.
The weighted estimates for the curl of $v_t$, $v_{ttt}$ and $\p_t^5v$ are obtained similarly.
\end{proof}

\subsection{Energy estimates}

We assume that we have smooth solutions $\eta$
on a time interval $[0,T ]$, and that for all such solutions, the time $T>0$ is taken sufficiently small so that for
$t \in [0,T ]$ 
\begin{equation}\label{assumptions_2d}
\begin{array}{l}
 {\frac{1}{2}} \le J(t) \le {\frac{3}{2}}\,, \\
\|  \eta(t)\|^2_ {3.5}   \le 2 | \Omega^+|^2 + 1 \,,   \\
 \| \p_t^a v(t)\|^2_ {3-a/2}  \le 2 \|\p_t^a v(0)\|^2_{3-a/2} + 1\, \ \  \  \text{ for } \ a=0,1,...,6 \,.
\end{array}
\end{equation}  
  The right-hand sides appearing in these inequalities shall be denoted by a
generic constant $C$ in the estimates appearing below.  Once we establish our a priori bounds, can indeed verify
that our solution adhere to the assumptions (\ref{assumptions_2d}) by means of the fundamental theorem of 
calculus.

\subsubsection{The structure of the estimates}   Due to the degeneracy of the initial density function $\rho_0$, one time derivative
scales like one-half of a space derivative.   The energy estimates for the time and tangential derivatives are obtained by first
studying the $\bar \p^4$-differentiated Euler equations, then the $\bar \p^3 \p_t^2$-differentiated Euler equations, and so on, until
we reach the $\bar\p^0 \p_t^8$-differentiated Euler equations.   The estimates for the normal derivatives are then found using
elliptic-type estimates.   The Sobolev embedding theorem requires that we use $H^4(\Omega)$ as the minimal regularity of
$\eta(t)$.

\subsubsection{The $\bar \p^4$-problem}
\begin{proposition} \label{prop1_2d} For $ \delta >0$ and letting the constant $M_0$ depend on $1/ \delta $, 
\begin{align}
&\sup_{t \in [0,T]} \left( \int_\Omega \rho_0(x) |\bar \p^4 v(x,t)|^2 dx + \int_\Omega \rho_0^2(x) | \bar \p^4 D\eta(x,t)|^2 dx\right) \nonumber \\
& \qquad\qquad\qquad\qquad\qquad\qquad\qquad
 \le M_0
+ \delta \sup_{t \in [0,T]} E(t) + C\, T\, P( \sup_{t \in [0,T]} E(t)) \,. \label{energy1_2d}
\end{align} 
\end{proposition}

\begin{proof}
  Letting $\bar \p^4$ act on $\rho_0 v_t^i + a^k_i (\rho_0^2 J^{-2} ),_k=0$, and taking the
$L^2(\Omega)$-inner product with $\bar \p^4 v^i$, we obtain 
\begin{align*} 
{\frac{1}{2}} {\frac{d}{dt}}  \int_\Omega &\rho_0 |\bar \p^4 v|^2 dx + \int_\Omega \bar \p^4 a^k_i (\rho_0^2 J^{-2} ),_k \bar \p^4 v^i dx 
+ \int_\Omega a^k_i (\rho_0^2 \bar \p^4 J^{-2} ),_k \bar \p^4 v^i dx \\
&= \sum_{l=1}^3 c_l \int_\Omega \bar \p^{4-l} a^k_i \, (\rho_0^2 \bar \p^l J^{-2} ),_k \, \bar \p^4 v^i \, dx \,.
\end{align*} 
Integrating the first term from $0$ to $t\in (0,T]$ produces the first term on the left-hand side of (\ref{energy1_2d}). 

We define the following three integrals
\begin{align*}
\mathcal{I} _1 &=   \int_\Omega \bar \p^4 a^k_i (\rho_0^2 J^{-2} ),_k \bar \p^4 v^i dx  \, \\
\mathcal{I}_2  & =  \int_\Omega a^k_i (\rho_0^2 \bar \p^4 J^{-2} ),_k \bar \p^4 v^i dx \, \\
\mathcal{R} & = \sum_{l=1}^3 c_l \int_\Omega \bar \p^{4-l} a^k_i \, (\rho_0^2 \bar \p^l J^{-2} ),_k \, \bar \p^4 v^i \, dx \,.
\end{align*} 
The last integral introduces our notation $ \mathcal{R} $ for the {\it remainder}, which throughout the paper will consist of integrals of lower-order terms
which can, via elementary inequalities together with our assumptions (\ref{assumptions_2d}),  easily be shown to satisfy the following estimate:
\begin{equation}\label{remainder}
\int_0^T  \mathcal{R} (t)dt \le M_0  + \delta \sup_{t \in [0,T]} E(t) + C\, T\, P ( \sup_{t \in [0,T]} E(t)) \,.
\end{equation} 

The sum  of $\int_0^T[ \mathcal{I}_1(t) + \mathcal{I} _2(t)]dt$ together with the estimates for $ \operatorname{curl} \eta$ given by Proposition \ref{curl_est} will
provide the remaining energy contribution $  \int_\Omega \rho_0^2(x,t) | \bar \p^4 D\eta|^2 dx$ plus error terms which have the
same bound as $ \mathcal{R} $.

\vspace{.1 in}
\noindent
{\bf Analysis of  $\int_0^T \mathcal{R} dt $.}  
We integrate by parts with respect to $x_k$ and then with respect to the time derivative $ \p_t$, and use (\ref{a3}) to obtain that
\begin{align*} 
\mathcal{R} &  = - \sum_{l=1}^3 c_l 
 \int_0^T \int_\Omega  \bar \p^{4-l} a^k_i \, \rho_0^2 \bar \p^l J^{-2}   \ \bar \p^4 v^i,_k \ dxdt \\
& =  \sum_{l=1}^3 c_l  \int_0^T \int_\Omega \rho_0\left( \bar \p^{4-l} a^k_i  \bar \p^l J^{-2}\right)_t  \rho_0\bar \p^4 \eta^i,_k dxdt
-        \sum_{l=1}^3 c_l  \int_\Omega \rho_0  \bar \p^{4-l} {a}^k_i  \bar \p^l  J^{-2}  \rho_0  \bar \p^4 \eta^i,_k dx \Bigr|_0^T \,.
\end{align*} 

Notice that when  $l=3$,  the integrand in the spacetime integral on the right-hand side scales like
$ \ell \  [  \bar \p D\eta \, \rho_0  \bar \p^3 \p_t J^{-2}  + \bar \p Dv \,  \rho_0 \bar \p^3 J^{-2}  ] \ \rho_0 \bar \p^4 D \eta$ where $\ell$ denotes an $L^ \infty (\Omega)$
function. Since $\| \rho_0 \p_t^2 J^{-2} (t)\|^2_3$ is contained in the energy function $E(t)$ and since $\bar \p D \eta(t) \in L^ \infty (\Omega)$,
the first summand is estimated using an $L^ \infty$-$L^2$-$L^2$ H\"{o}lder's inequality, while for  the second summand, we use that
$\| \rho_0  J^{-2} (t)\|^2_4$ is contained in $E(t)$ together with an $L^4$-$L^4$-$L^2$ H\"{o}lder's inequality.

When $l=1$, the integrand in the spacetime integral on the right-hand side scales like
$ \ell \  [  \bar \p D\eta \, \rho_0  \bar \p^3  {a_t}^k_i  + \bar \p Dv \,  \rho_0 \bar \p^3 a^k_i ] \ \rho_0 \bar \p^4 \eta^i,_k$.
 Since $\| \rho_0 \bar \p^3 Dv_t (t)\|^2_0$ is contained in the energy function $E(t)$ and since $\bar \p D \eta \in L^ \infty (\Omega)$,
the first summand is estimated using an $L^ \infty$-$L^2$-$L^2$ H\"{o}lder's inequality.   We write the second summand as
$$ \bar \p Dv \,  \rho_0 \bar \p^3 a^\beta_i  \ \rho_0 \bar \p^4 \eta^i,_\beta+ \bar \p Dv \,  \rho_0 \bar \p^3 a^3_i  \ \rho_0 \bar \p^4 \eta^i,_3.$$  
We estimate
\begin{align}
&
\int_0^T\int_\Omega   \bar \p Dv \,  \rho_0 \bar \p^3 a^\beta_i  \ \rho_0 \bar \p^4 \eta^i,_\beta \ dxdt  =-
\int_0^T\int_\Omega  [ \bar \p Dv \,  \rho_0 \bar \p^3 a^\beta_i,_\beta  \ \rho_0 \bar \p^4 \eta^i + \bar\p Dv,_\beta\,  \rho_0 \bar \p^3 a^\beta_i  \ \rho_0 \bar \p^4 \eta^i]
dxdt \nonumber \\
& \qquad\qquad
\le C \int_0^T \bigl( \| \bar \p  Dv(t)\|_{L^3(\Omega)} \| \rho_0 \bar \p^4 a (t) \|_0 \ \|\rho_0 \bar \p^4 \eta(t) \|_{L^6(\Omega)}\nonumber \\
& \qquad\qquad\qquad\qquad\qquad\qquad\qquad
+ \| \bar \p^2  Dv(t)\|_{L^3(\Omega)} \| \rho_0 \bar \p^4\eta(t) \|_{L^6(\Omega)} \|\bar \p^3 a \|_0 \bigr) dt \nonumber \\
& \qquad\qquad
\le C \int_0^T \bigl( \| \bar \p  Dv(t)\|_{H^{0.5}(\Omega)} \| \rho_0 \bar \p^4 a(t) \|_0 \ \|\rho_0 \bar \p^4 \eta(t) \|_1 \nonumber\\
& \qquad\qquad\qquad\qquad\qquad\qquad\qquad
+ \| \bar \p^2  Dv(t)\|_{H^{0.5}(\Omega)} \| \rho_0 \bar \p^4\eta(t) \|_1 \  \|\bar \p^3 a\|_0 \bigr) dt \nonumber \\
& \qquad\qquad
\le C \int_0^T \bigl( \|v(t)\|_{H^{3.5}(\Omega)} \| \rho_0 \bar \p^4 D\eta (t) \|_0^2 +   \|v(t)\|_{H^{2.5}(\Omega)} \| \rho_0 \bar \p^4 D\eta (t) \|_0 \|\eta(t) \|_4 \nonumber \\
& \qquad\qquad\qquad\qquad\qquad\qquad\qquad
+ \| v(t)\|_{H^{3.5}(\Omega)} \|  \eta(t) \|_4^2 \bigr) dt \,, \label{Rest_2d}
\end{align} 
where we have used H\"{o}lder's inequality, followed by the Sobolev embeddings 
$$
H^{0.5}(\Omega) \hookrightarrow L^3(\Omega) \ \ \text{ and } \ \ H^{1}(\Omega) \hookrightarrow L^6(\Omega) \,.
$$
We also rely on the interpolation estimate 
\begin{align} 
\|v\|^2_{L^2(0,T; H^{3.5}(\Omega))} & \le C\bigl(  \|v(t)\|_3 \|\eta\|_4 \bigr)\Bigr|^0_T +C \|v_t\|_{L^2(0,T; H^3(\Omega))}   \|\eta\|_{L^2(0,T; H^4(\Omega))} \nonumber \\
&\le M_0 + \delta \sup_{t \in [0,T]} \|\eta(t)\|_4^2+  C \, T\, \sup_{t \in [0,T]}  \Bigl(\|\eta(t)\|^2_4+ \|v_t(t)\|^2_3\Bigr) \,, \label{interp1}
\end{align} 
where the last inequality follows from Young's and  Jensen's inequalities.   Using this together with the Cauchy-Schwarz inequality,  (\ref{Rest_2d}) is
bounded by $C \, T\, P( \sup_{t \in [0,T]} E(t))$.   Next, since  (\ref{a3i}) shows that each component of $a^3_i$ is quadratic in $\bar \p \eta$, we see that the same analysis shows the
spacetime integral of $\bar \p Dv \,  \rho_0 \bar \p^3 a^3_i  \ \rho_0 \bar \p^4 \eta^i,_3$ has the same bound, and so we have estimated
the case $l=1$.  

 For the case that $l=2$, the integrand in the spacetime integral on the right-hand side  of the expression for $ \mathcal{R} $
scales like $ \ell \   \bar \p^2 D\eta \,  \bar \p^2 D v   \ \rho_0 \bar \p^4D\eta$, so that an $L^6-L^3-L^2$ H\"{o}lder's inequality, followed
by the same analysis as for the case $l=1$ provides the same bound as for the case $l=1$.

To deal with the space integral on the right-hand side of the expression for $ \mathcal{R} $, the integral at time $t=0$ is equal to zero
since $\eta(x,0)=x$, whereas the integral evaluated at $t=T$ is written, using the fundamental theorem of calculus, as
$$
-        \sum_{l=1}^3 c_l  \int_\Omega \rho_0  \bar \p^{4-l} {a}^k_i  \bar \p^l  J^{-2}  \rho_0  \bar \p^4 \eta^i,_k dx\Bigr|_{t=T}=
-        \sum_{l=1}^3 c_l  \int_\Omega \rho_0  \int_0^T (\bar \p^{4-l} {a}^k_i  \bar \p^l  J^{-2})_t  \rho_0  \bar \p^4 \eta^i,_k(T) dx
$$
which can be estimated in the identical fashion as the corresponding spacetime integral.   As such, we have shown that $ \mathcal{R} $
has the claimed bound (\ref{remainder}).

\vspace{.1 in}
\noindent
{\bf Analysis of the integral $ \mathcal{I} _1$.}  Because $\rho_0 =0$ on $\Gamma=\{ x_3=1\}$, we integrate by parts to find that
\begin{align*} 
\mathcal{I} _1 &=   -\int_\Omega \rho_0^2 J^{-2}\,  \bar \p^4 a^k_i  \bar \p^4 v^i,_k dx + \int_{\{x_3=0\}} \rho_0^2 J^{-2} \, \bar \p ^4 a^3_i \bar \p^4 v^i \, dx_1dx_2 \\
&=   -\int_\Omega \rho_0^2 J^{-2}\,  \bar \p^4 a^k_i  \bar \p^4 v^i,_k dx \,,
\end{align*} 
since on the fixed boundary $\{x_3=0\}$,  $\eta^3 = x_3$ so that according to (\ref{a3i}), the components $a^3_1=0$ and $a^3_2 =0$ on
$\{x_3=0\}$, and $v^3=0$ on $\{x_3=0\}$, so that
 $\bar \p^4 a^3_i  \bar \p^4v^i  =0 $ on$\{ x_3=0\}$.

To estimate ${{\mathcal I} _1}$,
we use the formula (\ref{a1}) for horizontally differentiating the cofactor matrix:
$$
{\mathcal{I} _1} =
\int_\Omega { \rho_0}^2  J  ^{-3} \,
\bar\p ^4 \eta^r,_s \, [a^s_i a^k_r - a^s_r a^k_i] 
\  \bar\p  ^4  v^i,_k \, dx   + \mathcal{R}\,,
$$
where the remainder $\mathcal{R}$ satisfies (\ref{remainder}).
We decompose the highest-order term in ${ \mathcal{I} _1}$ as the sum of the following two integrals:
\begin{align*}
{{\mathcal I}_1}_a & =  \int_\Omega  { \rho_0}^2 J  ^{-3} \
(\bar\p ^4 \eta^r,_s a^s_i)(\bar\p  ^4   v^i,_k  a^k_r) dx,  \\
{{\mathcal I}_1}_b & = -\int_\Omega { \rho_0}^2 J  ^{-3} \
(\bar\p ^4 \eta^r,_s a^s_r) (\bar\p  ^4  v^i,_k  a^k_i) dx  \,.  
\end{align*}
Since $v= \eta _t$, ${\mathcal{I} _1}_a$ is an exact derivative modulo an antisymmetric
commutation with respect to the free indices $i$ and $r$; namely,
\begin{equation}\label{newss0}
\bar\p ^4 \eta ^r,_s a^s_i \bar\p ^4 v^i,_k a^k_r
=\bar\p ^4 \eta ^i,_s a^s_r \bar\p ^4 v^i,_k a^k_r
+(\bar\p ^4 \eta ^r,_s a^s_i - \bar\p ^4 \eta ^i,_s a^s_r)\bar\p ^4 v^i,_k a^k_r \,.
\end{equation} 
 Using the notation 
 $$[D_\eta F]^i_r = a^s_r F^i,_s  \text{ for any vector field } F\,,$$
\begin{equation}\label{newss1}
\bar\p ^4 \eta ^i,_s a^s_r \bar\p ^4 v^i,_k a^k_r = {\frac{1}{2}} \frac{d}{dt} |D_\eta \bar\p ^4 \eta |^2
- {\frac{1}{2}} \bar\p ^4 \eta^r,_s\, \bar\p ^4\eta^i,_k  \ ( a^s_r a^k_i)_t \,,
\end{equation} 
so the first term on the right-hand side of (\ref{newss0}) produces an  exact derivative in time.

For the second term on the right-hand side of (\ref{newss0}), note the identity
\begin{equation}\label{newss2}
(\bar\p ^4 \eta ^r,_s a^s_i - \bar\p ^4 \eta ^i,_s a^s_r)\bar\p ^4 v^i,_k a^k_r
= -J^2 \varepsilon _{ijk} \bar \p ^4 \eta ^k,_r A^r_j \, \varepsilon _{imn}\bar \p ^4 v^n,_s A^s_m\,.
\end{equation} 
We have used the permutation symbol $\varepsilon$ to encode the anti-symmetry in this relation, and the basic
fact that the trace of the product of  symmetric and antisymmetric matrices is equal to zero.

Recalling our notation
$[{\operatorname{curl}} _\eta F]^i = \varepsilon_{ijk} F^k,_r A^r_j$,
(\ref{newss2}) can be written as
\begin{equation}\label{newss3}
(\bar\p ^4 \eta ^r,_s a^s_i - \bar\p ^4 \eta ^i,_s a^s_r)\bar\p ^4 v^i,_k a^k_r
= -J^2 \operatorname{curl} _\eta \bar\p ^4 \eta  \cdot \operatorname{curl} _\eta \bar\p ^4 v\,,
\end{equation} 
which can also be written as an exact derivative in time:
\begin{equation}\label{newss4}
\operatorname{curl} _\eta \bar\p ^4 \eta  \cdot \operatorname{curl} _\eta \bar\p ^4 v
= {\frac{1}{2}} \frac{d}{dt} | \operatorname{curl} _\eta \bar\p ^4 \eta |^2 - 
\bar\p ^4 \eta ^k,_r \bar\p ^4 \eta ^k,_s (A^r_j A^s_j)_t + 
\bar\p ^4 \eta ^k,_r \bar \p ^4 \eta ^j,_s (A^r_j A^s_k)_t \,.
\end{equation} 
The terms in (\ref{newss1}) and (\ref{newss4}) which are not the exact time derivatives are quadratic
in $\rho_0\bar \p ^4 D \eta $ with coefficients in $L^\infty ( [0,T] \times \Omega )$; denoting the
integral over $\Omega $ of
such terms by $\mathcal{Q}_{ \rho_0 \bar\p ^4 D \eta }$ ,
\begin{equation}\nonumber
{\mathcal{I} _1}_a = {\frac{1}{2}} \frac{d}{dt}  \int_\Omega { \rho_0}^2 J^{-3} |D_\eta \bar \p ^4
 \eta  |^2 dx
-
{\frac{1}{2}} \frac{d}{dt}  \int_\Omega { \rho_0}^2 \,   J^{-1} | \operatorname{curl} _\eta
\bar \p ^4\eta |^2 dx +\mathcal{Q}_{ \rho_0\bar \p ^4 D \eta } + \mathcal{R} \,,
\end{equation} 
where $\int_0^T |\mathcal{Q}_{ \rho_0 \bar \p ^4 D \eta }|\, dt \le C\,T\,P( \sup_{t \in [0,T]} E(t))$,
and $ \mathcal{R} $ satisfies (\ref{remainder}).

With the notation  $\operatorname{div} _\eta F = A^j_i F^i,_j$,
the differentiation formula (\ref{J1}) shows that  $ {\mathcal{I} _1}_b$ can be written as
\begin{equation}\nonumber
{\mathcal{I} _1}_b = 
-{\frac{1}{2}} \frac{d}{dt}  \int_\Omega { \rho_0}^2\,  J^{-1} | \operatorname{div} _\eta
\bar \p ^4\eta |^2 dx +\mathcal{Q}_{ \rho_0\bar\p ^4 D \eta } + \mathcal{R} \,.
\end{equation} 
It follows that
\begin{align*} 
\mathcal{I} _1& = {\frac{1}{2}} \frac{d}{dt}  \int_\Omega { \rho_0}^2 \left( J^{-3} | D_\eta \bar \p ^4 \eta |^2 -  J^{-1} | \operatorname{curl} _\eta\bar \p ^4\eta |^2 
-J^{-1} | \operatorname{div} _\eta \bar\p ^4\eta |^2 \right) dx + \mathcal{R} \\
&= {\frac{1}{2}} \frac{d}{dt}  \int_\Omega { \rho_0}^2 \left( J^{-3} | D \bar \p ^4 \eta |^2 -  J^{-1} | \operatorname{curl} _\eta\bar \p ^4\eta |^2 
-J^{-1} | \operatorname{div} _\eta \bar\p ^4\eta |^2 \right) dx + \mathcal{R}\,,
\end{align*} 
where we have used the fundamental theorem of calculus for the second equality on the term $D_\eta \bar\p^4 \eta$.

\vspace{.1 in}
\noindent
{\bf Analysis of the integral $ \mathcal{I} _2$.}
Integration by parts once again yields
$$
\mathcal{I}_2  =  - \int_\Omega \rho_0^2 \bar \p^4 J^{-2} \,  a^k_i   \bar \p^4 v^i,_k dx \,.
$$
Since $\bar \p^4 J^{-2} = - 2 J^{-3} \bar \p^4 J$ plus lower-order terms, which have  at most three horizontal derivatives acting on
$J$.  For such lower-order terms, we integrate by parts with respect to $\p_t$, and estimate the resulting integrals in
the same manner as we estimated the remainder term $ \mathcal{R} $, and obtain the same bound.

Thus, 
\begin{align*} 
\mathcal{I} _2 &= 2 \int_\Omega \rho_0^2 J^{-3} a^r_s \bar \p^4 \eta^s,_r\ a^k_i   \bar \p^4 v^i,_k dx  + \mathcal{R}  \\
&=  {\frac{d}{dt}} \int_\Omega \rho_0^2 J^{-3} a^r_s \bar \p^4 \eta^s,_r\ a^k_i   \bar \p^4 \eta^i,_k dx  
-  \int_\Omega \rho_0^2 (J^{-3} a^r_sa^k_i )_t  \ \bar \p^4 \eta^s,_r \bar \p^4 \eta^i,_k dx + \mathcal{R} 
\end{align*} 
Given our identities for differentiating $a$ and $J$, the Sobolev embedding theorem together with our assumptions (\ref{assumptions_2d}) and
the Cauchy-Schwarz inequality show that
$$
\int_0^T\int_\Omega \rho_0^2 (J^{-3} a^r_sa^k_i )_t  \ \bar \p^4 \eta^s,_r \bar \p^4 \eta^i,_k dxdt \le C \,T\, \sup_{t \in [0,T]} E(t) \,;
$$
consequently, we can write
\begin{equation}\label{szz1_2d}
  \int_\Omega \rho_0^2 J^{-3} a^r_s \bar \p^4 \eta^s,_r\ a^k_i   \bar \p^4 \eta^i,_k dx   = M_0 + \int_0^t [\mathcal{I} _2(t')+ \mathcal{R} (t')] dt'    \,.
\end{equation} 

On the other hand,
\begin{align} 
 \int_\Omega& \rho_0^2 J^{-3} a^r_s \bar \p^4 \eta^s,_r\ a^k_i   \bar \p^4 \eta^i,_k dx \nonumber   \\
 & = 
 \int_\Omega \rho_0^2 J^{-2} \bigl( \bar \p^4 \operatorname{div} \eta +   \bar \p^4 \eta^s,_r \int_0^t {a_t}^r_s dt' \bigr) \ 
 \bigl( \bar \p^4 \operatorname{div} \eta +   \bar \p^4 \eta^i,_k \int_0^t {a_t}^k_i dt' \bigr)dx  \nonumber \\
 & = 
 \int_\Omega \rho_0^2 J^{-2} |\bar \p^4 \operatorname{div} \eta |^2 dx +   2\int_\Omega \rho_0^2 J^{-2} \bar \p^4 \operatorname{div} \eta\  \bar \p^4 \eta^s,_r \int_0^t {a_t}^r_s dt'  \, dx \nonumber \\
 &
 \qquad \qquad \qquad
 +  \int_\Omega \rho_0^2 J^{-2}\  \bar \p^4 \eta^s,_r \int_0^t {a_t}^r_s dt'  \
 \bar \p^4 \eta^i,_k \int_0^t {a_t}^k_i dt'  \ dx   \label{szz2_2d}
\end{align} 
Yet another application of the Sobolev embedding theorem together with our assumptions (\ref{assumptions_2d}) and
the Cauchy-Schwarz inequality shows that the second and third integrals on the right-hand side are bounded by $M_0 +
C\, T\, \sup_{t \in [0,T]} E(t)$, so that combining (\ref{szz1_2d}) and (\ref{szz2_2d}), we find that
\begin{equation}\label{4_2d}
 \int_\Omega \rho_0^2 J^{-2} |\bar \p^4 \operatorname{div} \eta |^2 dx  = M_0 + \int_0^t [\mathcal{I} _2(t')  + \mathcal{R}(t')] dt'   \,.
\end{equation}

\vspace{.1 in}
\noindent
{\bf Summing inequalities.}  
Summing (\ref{4_2d}) with  $ \mathcal{I} _1$  yields 
\begin{align*} 
&\sup_{t \in [0,T]}  {\frac{1}{2}} 
 \Bigl[ \int_\Omega \rho_0^2J^{-2}  | \bar \p^4 D\eta|^2dx  +\int_\Omega \rho_0^2J^{-2}  | \bar \p^4  \operatorname{div} \eta|^2dx  
- \int_\Omega \rho _0^2J^{-2}  | \bar \p^4 \operatorname{curl} \eta|^2  dx\Bigr] \\
& \qquad\qquad\qquad\qquad\qquad  \le M_0 +  \delta \sup_{t \in [0,T]} E(t) + C\,T\, P( \sup_{t \in [0,T]} E(t)) \,.
\end{align*} 
Adding to this,  the inequality (\ref{curl_estimate}), and possible readjusting our
constants, we obtain the desired result, and complete the proof of the proposition.
\end{proof}

Since  $\eta \cdot T_ \alpha  = \eta^\alpha $ for $ \alpha =1,2$,  we have the following
\begin{corollary}\label{cor1_2d} For $ \alpha =1,2$, 
$$ \sup_{t \in [0,T]}  |\eta^\alpha |^2_{3.5}  \le M_0 + C\, T\, P
( \sup_{t \in [0,T]} E(t)) \,.$$
\end{corollary}
\begin{proof}
The weighted embedding estimate (\ref{w-embed}) shows that
$$
\| \bar \p^4 \eta\|^2_0 \le C  \int_\Omega \rho_0^2 \bigl( |\bar \p^4 \eta|^2 + |\bar\p^4D\eta|^2 \bigr) dx \,.
$$

Now 
$$\sup_{t \in [0,T]} \int_\Omega \rho_0^2 |\bar \p^4 \eta|^2 dx = \sup_{t \in [0,T]} \int_\Omega \rho_0^2 \left| \int_0^t \bar \p^4 v dt'\right|^2 dx \le T^2 \sup_{t \in [0,T]} \| \sqrt{\rho_0} \bar \p^4 v \|^2_0$$
It follows from Proposition \ref{prop1_2d} that
$$
\sup_{t \in [0,T]} \| \bar \p^4 \eta\|^2_0 \le M_0 + C\,T\, P( \sup_{t \in [0,T]} E(t)) \,.
$$
According to our curl estimates (\ref{curl_estimate}), $ \sup_{t \in [0,T]} \| \operatorname{curl} \eta\|^2_3 \le M_0 + C\,T\, P( \sup_{t \in [0,T]} 
E(t))$, from which it follows that
$$
\sup_{t \in [0,T]} \| \bar \p^4 \operatorname{curl} \eta\|^2_{H^1(\Omega)'} \le M_0 + C\,T\, P( \sup_{t \in [0,T]} E(t)) \,,
$$
since $\bar \p$ is a horizontal derivative, and integration by parts with respect to $\bar \p$ does not produce any boundary contributions.
From the tangential trace inequality (\ref{tangentialtrace}), we find that
$$
\sup_{t \in [0,T]} | \bar \p^4 \eta^\alpha |^2_{-1/2} \le M_0 + C\,T\,P( \sup_{t \in [0,T]} E(t)) \,,
$$
from which the assertion of the corollary follows.
\end{proof}

\subsection{The $\p_t^8$-problem}
\begin{proposition} \label{prop5_2d} For $ \delta >0$ and letting the constant $M_0$ depend on $1/ \delta $, 
\begin{align}
&\sup_{t \in [0,T]} \left( \int_\Omega \rho_0 |\p_t^8 v(x,t)|^2 dx + \int_\Omega \rho_0^2(x,t) |  \p_t^7 Dv(x,t)|^2 dx\right) \nonumber \\
& \qquad\qquad\qquad\qquad\qquad\qquad\qquad
 \le M_0
+ \delta \sup_{t \in [0,T]} E(t) + C\, T\, P( \sup_{t \in [0,T]} E(t)) \,. \label{energy5_2d}
\end{align} 
\end{proposition}

\begin{proof}
  Letting $\p_t^8$ act on $\rho_0 v_t^i + a^k_i (\rho_0^2 J^{-2} ),_k=0$, and taking the
$L^2(\Omega)$-inner product with $\p_t^8 v^i$, we obtain 
\begin{align*} 
{\frac{1}{2}} {\frac{d}{dt}}  \int_\Omega &\rho_0 |\p_t^8 v|^2 dx + \int_\Omega \p_t^8 a^k_i (\rho_0^2 J^{-2} ),_k \p_t^8 v^i dx 
+ \int_\Omega a^k_i (\rho_0^2 \p_t^8 J^{-2} ),_k \p_t^8 v^i dx \\
&= \sum_{l=1}^7 c_l \int_\Omega  \p_t^{8-l} a^k_i \, (\rho_0^2 \p_t^l J^{-2} ),_k \, dx \,.
\end{align*} 
Integrating the first term from $0$ to $t\in (0,T]$ produces the first term on the left-hand side of (\ref{energy5_2d}). 

We define the following three integrals
\begin{align*}
\mathcal{I} _1 &=   \int_\Omega\p_t^8 a^k_i (\rho_0^2 J^{-2} ),_k  \p_t^8 v^i dx  \, \\
\mathcal{I}_2  & =  \int_\Omega a^k_i (\rho_0^2 \p_t^8 J^{-2} ),_k \p_t^8 v^i dx \, \\
\mathcal{R} & = \sum_{l=1}^7 c_l \int_\Omega \p_t^{8-l} a^k_i \, (\rho_0^2 \p_t^l J^{-2} ),_k \, \p_t^8v^i \, dx \,.
\end{align*} 
The sum  of $\int_0^T[ \mathcal{I}_1(t) + \mathcal{I} _2(t)]dt$ together with the {curl} estimates given by Proposition \ref{curl_est} will
provide the remaining energy contribution $  \int_\Omega \rho_0^2(x,t) |  \p_t^7 Dv|^2 dx$ plus error terms which have the
same bound as $ \mathcal{R} $.

\vspace{.1 in}
\noindent
{\bf Analysis of  $\int_0^T \mathcal{R} dt $.}  
We integrate by parts with respect to $x_k$ and then with respect to the time derivative $ \p_t$ to obtain that
\begin{align*} 
\mathcal{R} &  = - \sum_{l=1}^7 c_l 
 \int_0^T \int_\Omega  \p_t^{8-l} a^k_i \, \rho_0^2 \p_t^l J^{-2}   \ \p_t^8 v^i,_k \ dxdt \\
& =  \sum_{l=1}^7 c_l  \int_0^T \int_\Omega \rho_0\left( \p_t^{8-l} a^k_i  \p_t^l J^{-2}\right)_t  \rho_0\p_t^7 v^i,_k dxdt
-        \sum_{l=1}^7 c_l  \int_\Omega \rho_0  \p_t^{8-l} {a}^k_i  \bar \p^l  J^{-2}  \rho_0  \p_t^7 v^i,_k dx \Bigr|_0^T \,.
\end{align*} 

Notice that when  $l=7$,  the integrand in the spacetime integral on the right-hand side scales like
$ \ell \   [ Dv_t \, \rho_0 \p_t^6Dv + Dv \, \rho_0 \p_t^7Dv ]  \, \rho_0 \p_t^7D v$ where $\ell$ denotes an $L^ \infty (\Omega)$
function. 
Since $\| \rho_0 \p_t^7 Dv (t)\|^2_0$ is contained in the energy function $E(t)$, $ D v_t(t)$ is bounded in $L^ \infty (\Omega)$,
and since we can write
$ \rho_0 \p_t^6 Dv(t) = \rho_0 \p_t^6 Dv(0) + \int_0^t \rho_0 \p_t^7 Dv(t')dt'$,
the first  and second summands are both estimated using an $L^ \infty$-$L^2$-$L^2$ H\"{o}lder's inequality.

The case $l=6$ is estimated exactly the same way as the case $l=3$ in the proof of Proposition \ref{prop1_2d}.   For the case
$l=5$, the integrand in the spacetime integral scales like
$\ell [ Dv_{tt} \rho_0 \p_t^6 J^{-2}  + Dv_{ttt}  \rho_0 Dv_{tttt}] \rho_0 \p_t^7 Dv$.  Both summands can be estimated using 
an $L^ 3$-$L^6$-$L^2$ H\"{o}lder's inequality.  The case $l=4$ is treated as the case $l=5$.  The case $l=3$ is also treated in
the same way as $l=5$.  The case $l=2$ is estimated exactly the same way as the case $l=1$ in the proof of Proposition \ref{prop1_2d}.
The case $l=1$ is treated in the same way as the case $l=7$.

To deal with the space integral on the right-hand side of the expression for $ \mathcal{R} $, the integral at time $t=0$ is bounded by $M_0$, whereas the integral evaluated at $t=T$ is written, using the fundamental theorem of calculus, as
\begin{align*} 
\sum_{l=1}^7 c_l  \int_\Omega \rho_0  \p_t^{8-l} {a}^k_i  \p_t^l  J^{-2}  \rho_0 \p_t^7 v^i,_k dx\Bigr|_{t=T} & =
\sum_{l=1}^7 c_l \int_\Omega \rho_0  \p_t^{8-l} {a}^k_i(0)  \p_t^l  J^{-2} (0) \rho_0 \p_t^7 v^i,_k(T) dx\\
& +  \sum_{l=1}^7 c_l  \int_\Omega \rho_0  \int_0^T (\p_t^{8-l} {a}^k_i  \p_t^l  J^{-2})_t dt'\,  \rho_0  \p_t^7v^i,_k(T) dx.
\end{align*} 
The first integral on the right-hand side is estimated using Young's inequality, and is bounded by $M_0 + \delta \sup_{t \in [0,T]} E(t)$,
while the second integral 
can be estimated in the identical fashion as the corresponding spacetime integral.   As such, we have shown that $ \mathcal{R} $
has the claimed bound (\ref{remainder}).

\vspace{.1 in}
\noindent
{\bf Analysis of the integral $ \mathcal{I} _1$.} As to the term ${\mathcal I} _1$,  using the identity (\ref{a2}),
the same computation as for the $\bar \p^4$-differentiated problem shows that
\begin{align*}
\rho_0^{2 }( \p_t^7 v^r,_s A^s_i)
\  (\p_t^8  v^i,_k A^k_r)
& = 
\frac{1}{2}\frac{d}{dt} |{\rho_0 } D_\eta \p_t^7 v(t)|^2
-
\frac{1}{2}\frac{d}{dt} |{\rho_0 }  {\operatorname{curl}} _\eta \p_t^7 v(t)|^2 \\
&  \qquad \qquad + {\frac{1}{2}}  
{ \rho_0}^2\p_t^7 v^k,_r \p_t^7 v^b,_s
(A^r_j A^s_m)_t [ \delta^j_m\delta^k_b - \delta^j_b \delta ^k_m] \,,
\end{align*}
and
\begin{align*}
-\rho_0^2(\p_t^7  v^r,_s\, A^s_r) \ (\p_t^8  v^i,_k  A^k_i) 
& = -{\frac{1}{2}} {\frac{d}{dt}} 
|{\rho_0} {\operatorname{div}} _\eta \p_t ^7 v|^2  
 + {\frac{1}{2}} 
 \rho_0^{2 } \p_t^7  v^r,_s \p_t^7  v^i,_k
(A^s_r A^k_i)_t  \,,
\end{align*}
and hence
\begin{align*} 
\mathcal{I} _1 = {\frac{1}{2}} \frac{d}{dt}  \int_\Omega { \rho_0}^2 \left( J^{-3} |D_\eta \p_t^7
 v  |^2 -  J^{-1} | \operatorname{curl} _\eta \p_t^7 v|^2 
-J^{-1} | \operatorname{div} _\eta \p_t^7 v|^2 \right) dx + \mathcal{R} \,.
\end{align*}

\vspace{.1 in}
\noindent
{\bf Analysis of the integral $ \mathcal{I} _2$.}
Integration by parts once again yields
$$
\mathcal{I}_2  =  - \int_\Omega \rho_0^2 \p_t^8 J^{-2} \,  a^k_i  \p_t^8 v^i,_k dx \,.
$$
Since $\p_t^8 J^{-2} = - 2 J^{-3} \p_t^8 J$ plus lower-order terms, which have at most seven time derivatives on $J$, and can be 
estimated in the same fashion as the remainder term $ \mathcal{R} $ above.

We see that
\begin{align*} 
\mathcal{I} _2 &= 2 \int_\Omega \rho_0^2 J^{-3} a^r_s \p_t^7v^s,_r\ a^k_i   \p_t^8 v^i,_k dx  + \mathcal{R}  \\
&=  {\frac{d}{dt}} \int_\Omega \rho_0^2 J^{-3} a^r_s \p_t^7 v^s,_r\ a^k_i   \p_t^7 v^i,_k dx  
-  \int_\Omega \rho_0^2 (J^{-3} a^r_sa^k_i )_t  \ \p_t^7 v^s,_r \p_t^7v^i,_k dx + \mathcal{R} 
\end{align*} 
Following our analysis of the term $ \mathcal{I} _2$ in the $\bar \p^4$-problem, we see that
 \begin{align} 
 \int_\Omega \rho_0^2 J^{-3} a^r_s \p_t^7v^s,_r\ a^k_i   \p_t^7 v^i,_k dx   = M_0 + \int_0^t [\mathcal{I} _2(t')+ \mathcal{R}(t')]dt'   \,.
 \label{szz3_2d}
\end{align} 
On the other hand,
\begin{align} 
 \int_\Omega& \rho_0^2 J^{-3} a^r_s\p_t^7 v^s,_r\ a^k_i   \p_t^7 v^i,_k dx  \nonumber \\
 & = 
 \int_\Omega \rho_0^2 J^{-2} \bigl( \p_t^7 \operatorname{div} v +   \p_t^7 v^s,_r \int_0^t {a_t}^r_s dt' \bigr) \ 
 \bigl( \p_t^7 \operatorname{div} v +  \p_t^7 v^i,_k \int_0^t {a_t}^k_i dt' \bigr)dx  \nonumber  \\
 & = 
 \int_\Omega \rho_0^2 J^{-2} |\p_t^7 \operatorname{div} v |^2 dx +  2 \int_\Omega \rho_0^2 J^{-2} \p_t^7 \operatorname{div} v\  \p_t^7 v^s,_r \int_0^t {a_t}^r_s dt'  \, dx  \nonumber \\
 &
 \qquad \qquad \qquad
 +  \int_\Omega \rho_0^2 J^{-2}\  \p_t^7 v^s,_r \int_0^t {a_t}^r_s dt'  \ \p_t^7 v^i,_k \int_0^t {a_t}^k_i dt'  \ dx   \label{szz4_2d}
\end{align} 
Yet another application of the Sobolev embedding theorem together with our assumptions (\ref{assumptions_2d}) and
the Cauchy-Schwarz inequality shows that the second and third integrals on the right-hand side are bounded by $M_0 +
C \sup_{t \in [0,T]} E(t)$, so that summing (\ref{szz3_2d}) and (\ref{szz4_2d}) shows that
\begin{equation}\label{14_2d}
\int_\Omega \rho_0^2 J^{-2} |\p_t^7 \operatorname{div} v |^2 dx  = M_0 + \int_0^t [ \mathcal{I} _2(t') + \mathcal{R}(t')]  dt'  \,.
\end{equation} 
\vspace{.1 in}
\noindent
{\bf Summing inequalities.}   
Summing (\ref{14_2d}) with $ \mathcal{I} _1$  yields  
\begin{align*} 
&\sup_{t \in [0,T]}  {\frac{1}{2}} 
 \Bigl[ \int_\Omega \rho_0^2J^{-2}  | \p_t^7 Dv|^2dx  + \rho_0^2J^{-2}  | \p_t^7  \operatorname{div} v|^2dx  
- \int_\Omega \rho _0^2J^{-2}  | \p_t^7 \operatorname{curl} v|^2  dx\Bigr] \\ 
& \qquad\qquad\qquad\qquad\qquad \le M_0 +  \delta \sup_{t \in [0,T]} E(t) + C\,T\, P( \sup_{t \in [0,T]} E(t)) \,.
\end{align*} 
Adding the curl estimate (\ref{curl_estimate}),  readjusting our
constants, we obtain the desired result, and complete the proof of the proposition.
\end{proof}

\subsection{The $\p_t^2 \bar \p^3$, $\p_t^4 \bar \p^2$, and $\p_t^6\bar \p$ problems}  Since we have provided detailed proofs
of the energy estimates for the two end-point cases of all space derivatives, the $\bar \p^4$ problem, and all time derivatives, the
$\p_t^8$ problem, we have covered all of the estimation strategies for all possible error terms in the three remaining
intermediated problems; meanwhile, the energy contributions for the three intermediate are found in the identical fashion as
for the $\bar \p^4$ and $\p_t^8$ problems.  As such we have the additional estimate
\begin{proposition} \label{additional_2d} For $ \delta >0$ and letting the constant $M_0$ depend on $1/ \delta $,  for $ \alpha =1,2$,
\begin{align*} 
& \sup_{t \in [0,T]} \sum_{a=1}^3 \Bigl[ | \p_t^{2a} \eta^\alpha (t)|_{3.5-a}^2 + \| \sqrt{\rho_0} \bar \p^{4-a} \ \p_t^{2a} v(t)\|_0^2
 + \| \rho_0 \bar \p^{4-a} \ \p_t^{2a} D\eta(t)\|^2_0\Bigr] \\
 &\qquad\qquad  \le M_0 + \delta \sup_{t \in [0,T]} E(t) + C\,T\, P( \sup_{t \in [0,T]} E(t)) 
\end{align*} 
\end{proposition}

\subsection{Additional elliptic-type estimates for normal derivatives}
Our energy estimates provide  a priori control of horizontal and time derivatives of $\eta$; it remains to gain a priori control of the
normal (or vertical) derivatives of $\eta$.   This is accomplished via a bootstrapping procedure relying on having $\p_t^7 v(t)$ bounded in
$L^2(\Omega)$.

\begin{proposition} \label{pt5v_2d} For $t \in [0,T]$, $\p_t^5 v(t) \in H^1(\Omega)$, $ \rho_0 \p_t^6 J^{-2} (t) \in H^1(\Omega)$ and
$$
\sup_{t \in [0,T]} \left( \| \p_t^5 v(t)\|^2_1 +  \| \rho_0 \p_t^6 J^{-2} (t)\|^2_1 \right) \le 
M_0 + \delta  {\sup_{t\in[0,T]}} E(t) + C \, T\, P({\sup_{t\in[0,T]}} E(t)) \,.
$$
\end{proposition}
\begin{proof}
We  write (\ref{ce0.a}) as $ v^i_t + 2 A^k_i (\rho_0 J^{-1} ),_k =0$, which we rewrite as 
\begin{equation}\label{eq400_2d}
v^i_t + \rho_0 a^k_i J ^{-2},_k - 2 a^3_i J^{-2} =0 \,.
\end{equation} 
We have used the fact that $ \rho_0,_ \beta =0$  for $\beta=1,2$, and
$ \rho_0 ,_3=-1$.  Letting $\p_t^6$ act on equation (\ref{eq400_2d}), we have that
\begin{align*} 
\rho_0a^3_i \p_t^6 J^{-2} ,_3 - 2 a^3_i \p_t^6 J^{-2}  
&= - \p_t^7 v^i - \rho_0\p_t^6 (a^\beta_iJ^{-2} ,_\beta  ) - (\p_t^6 a^3_i)[ -2J^{-2}  + \rho_0 J^{-2} ,_3] \\
& \qquad\qquad  + \sum_{a=1}^5 c_a \p_t^a a^3_i \p_t^{6-a} [ -2 J^{-2}  + \rho_0J^{-2} ,_3] \,.
\end{align*} 

According to Proposition \ref{prop5_2d} and \ref{additional_2d}, 
$$
 \sup_{t \in [0,T]} \Bigl( \| \p_t^7 v(t)\|^2_0 +  \|\rho_0 \bar \p D \p_t^5 v(t)\|^2_0\Bigr) \le M_0 + \delta \sup_{t \in [0,T]} E(t) + C\,T\, P( \sup_{t \in [0,T]} 
 E(t)) \,,
$$
and since (\ref{a3i}) shows that $a^3_i $ is quadratic in $\bar \p\eta$, we see that for all $t \in [0,T]$,
$$
 \left\|  [\rho_0 a^3_i \p_t^6 J^{-2} ,_3 - 2 a^3_i \p_t^6J^{-2}  ](t)
\right\|^2_0 \le M_0 + \delta  {\sup_{t\in[0,T]}} E(t) + C \, T\, P({\sup_{t\in[0,T]}} E(t)) \,.  
$$
It follows that 
\begin{align*} 
& \|  \rho_0 |a^3_ \cdot| \p_t^6 J^{-2} ,_3 (t) \|^2_0
 + 4\| |a^3_ \cdot| \, \p_t^6 J^{-2}  (t)\|^2_0 - 4 \int_ \Omega  \rho_0 | a^3_ \cdot |^2 \p_t^6 J^{-2}  \p_t^6J^{-2} ,_3 \, dx 
\\
& \qquad \qquad \qquad
\le
M_0 + \delta  {\sup_{t\in[0,T]}} E(t) + C \, T\, P({\sup_{t\in[0,T]}} E(t)) \,.
\end{align*} 
We assume that our solution is sufficiently smooth so that $ \p_t^6 J^{-2}  (t) \in H^2( \Omega)$, and in particular
that  $ [(\p_t^6 J^{-2}  )^2],_3$ is well-defined and integrable.  As such,
we write\footnote{Jang \& Masmoudi \cite{JaMa2008b} have counterexamples to the obtained inequality when $ J^{-2} $ is not 
sufficiently smooth.  It is important that the function $ J^{-2} $ has greater regularity  than the desired a priori estimate indicates, and
in particular, as we noted, $[(\p_t^6 J^{-2}  )^2],_3$ must be well-defined and integrable.}
\begin{align*} 
-4 \int_\Omega \rho_0 | a^3_ \cdot |^2 \p_t^6J^{-2}  \p_t^6 J^{-2} ,_3 \, dx 
&=
-2  \left\|  |a^3_ \cdot |\, \p_t^6 J^{-2} (t)\right\|^2_0
+2 \int_\Omega \rho_0 (| a^3_ \cdot |^2 ),_3 (\p_t^6 J^{-2}  )^2 \, dx  \\
& \qquad \qquad + 4 \int_{\{x_3=0\}} | \p_t^6 J^{-2} |^2 dx_1dx_2 \,,
\end{align*} 
so that together with our previous inequality,
\begin{align*} 
& \| \rho_0 \p_t^6 J^{-2},_3 (t)\|^2_0 + \| \p_t^6 J^{-2} (t)\|^2_0
\\
& \qquad \qquad \qquad
\le
M_0 + \delta  {\sup_{t\in[0,T]}} E(t) + C \, T\, P({\sup_{t\in[0,T]}} E(t)) 
+C \int_{\Omega} \rho_0 |\p_t^6 J^{-2} |^2 \, dx \,.
\end{align*} 
Since $\rho_0 \bar \p \p_t^6 J^{-2} (t)$ is already estimated by Proposition \ref{additional_2d}, then
\begin{align*} 
& \| \rho_0 \p_t^6 J^{-2}(t)\|^2_1 + \| \p_t^6 J^{-2} (t)\|^2_0
\\
& \qquad \qquad \qquad
\le
M_0 + \delta  {\sup_{t\in[0,T]}} E(t) + C \, T\, P({\sup_{t\in[0,T]}} E(t)) 
+C \int_{\Omega} \rho_0 |\p_t^6 J^{-2} |^2 \, dx \,.
\end{align*} 

We use Young's inequality and the fundamental theorem of calculus (with respect to $t$) for the last integral to find that for $ \delta >0$
\begin{align*} 
C\int_{\Omega^+} \rho_0 \p_t^6 J^{-2} \, \p_t^6 J^{-2} \, dx 
& \le \delta \left\| \p_t^6 J^{-2} (t)\right\|^2_0 + C_ \delta  \left\|  \rho_0 \p_t^5 D v  (t)\right\|^2_0 + M_0 + C\,T\, P( \sup_{t \in [0,T]} E(t))  \\
& \le \delta \left\| \p_t^6 J^{-2} (t)\right\|^2_0 + M_0 +  C\,T\, P( \sup_{t \in [0,T]} E(t)) \,,
\end{align*} 
where we have used the fact that $ \| \rho_0 \p_t^7 Dv(t)\|^2_0$ is contained in the energy function $E(t)$.  By once again
readjusting the constants, we see that on $[0,T]$
\begin{align} 
& \| \rho_0 \p_t^6 J^{-2} (t)\|^2_1
 + \left\| \p_t^6 J^{-2} (t)\right\|^2_0
\le
M_0 + \delta  {\sup_{t\in[0,T]}} E(t) + C \, T\, P({\sup_{t\in[0,T]}} E(t))  \,. \label{eq402_2d}
\end{align} 
With $J_t= a^j_i v^i,_j$, we see that
$$
a^j_i \p_t^5 v^i_j= \p_t^6 J - v^i,_j \p_t^5 a^j_i - \sum_{a=1}^4 c_a \p_t^a a^j_i \p_t^{5-a} v^i,_j
$$
so that using (\ref{eq402_2d}) together with  the fundamental theorem of calculus the estimate for the last two terms on the right-hand side,
we see that
\begin{align} 
\left\| a^j_i \p_t^5v^i,_j (t)\right\|^2_0
\le
M_0+ \delta  {\sup_{t\in[0,T]}} E(t) + C \, T\, P({\sup_{t\in[0,T]}} E(t))  \,, \nonumber
\end{align}
from which it follows that 
\begin{align} 
\left\| \operatorname{div} \p_t^5v (t)\right\|^2_0
\le
M_0 + \delta  {\sup_{t\in[0,T]}} E(t) + C \, T\, P({\sup_{t\in[0,T]}} E(t))  \,. \nonumber
\end{align}
According to Proposition \ref{curl_est}, 
$\| \operatorname{curl} \p_t^5v (t)\|^2_0\le M_0+  C \, T\, P({\sup_{t\in[0,T]}} E(t))$ and with the bound on $\p_t^5v ^\alpha  $ given
by  Proposition \ref{additional_2d}, Proposition \ref{prop1}  provides the estimate
\begin{align*} 
\left\|  \p_t^5v (t)\right\|^2_1
\le M_0 + \delta {\sup_{t\in[0,T]}} E(t) + C \, T\, P({\sup_{t\in[0,T]}} E(t))  \,.
\end{align*}
\end{proof}

Having a good bound for $\p_t^5v(t)$ in $H^1(\Omega)$ we proceed with our bootstrapping.   We let $\p_t^4$ act on equation (\ref{eq400_2d}), so that
\begin{align*} 
\rho_0a^3_i \p_t^4 J^{-2} ,_3 - 2 a^3_i \p_t^4 J^{-2}  
&= - \p_t^5 v^i - \rho_0\p_t^4 (a^\beta_iJ^{-2} ,_\beta  ) - (\p_t^4 a^3_i)[- 2J^{-2}  + \rho_0 J^{-2} ,_3] \\
& \qquad\qquad  + \sum_{a=1}^3 c_a \p_t^a a^3_i \p_t^{4-a} [ -2 J^{-2}  + \rho_0J^{-2} ,_3] \,,
\end{align*} 
with the right-hand side bounded  in $H^1(\Omega)$ by $M_0 + \delta {\sup_{t\in[0,T]}} E(t) + C \, T\, P({\sup_{t\in[0,T]}} E(t))$.
Using the argument   just given, we  conclude that 
$$
\sup_{t \in [0,T]} \bigl( \| v_{ttt}(t) \|^2_2 +  \| \rho_0 \p_t^4 J^{-2} (t)\|^2_2\bigr) \le 
M_0 + \delta {\sup_{t\in[0,T]}} E(t) + C \, T\, P({\sup_{t\in[0,T]}} E(t)) \,.
$$
Next, we let $\p_t^2$ act on equation (\ref{eq400_2d}), so that
\begin{align*} 
\rho_0a^3_i \p_t^2 J^{-2} ,_3 - 2 a^3_i \p_t^2 J^{-2}  
&= - \p_t^3 v^i - \rho_0\p_t^2 (a^\beta_iJ^{-2} ,_\beta  ) - (\p_t^2 a^3_i)[ -2J^{-2}  + \rho_0 J^{-2} ,_3] \\
& \qquad\qquad  + 2( \p_t a^3_i ) \p_t [ -2 J^{-2}  + \rho_0J^{-2} ,_3] \,,
\end{align*} 
with the right-hand side bounded  in $H^2(\Omega)$ by $M_0 + \delta {\sup_{t\in[0,T]}} E(t) + C \, T\, P({\sup_{t\in[0,T]}} E(t))$.
We conclude that
$$
\sup_{t \in [0,T]} \bigl( \| v_t(t) \|^2_3 +  \| \rho_0 \p_t^2 J^{-2} (t)\|^2_3\bigr) \le 
M_0 + \delta {\sup_{t\in[0,T]}} E(t) + C \, T\, P({\sup_{t\in[0,T]}} E(t)) \,.
$$
Finally, this estimate together with equation (\ref{eq400_2d}) shows that
\begin{align*} 
\rho_0a^3_i J^{-2} ,_3 - 2 a^3_i  J^{-2}  
&= -  v_t^i - \rho_0 a^\beta_iJ^{-2} ,_\beta   \,,
\end{align*} 
from which it follows that
\begin{equation}\label{ss777_2d}
\sup_{t \in [0,T]} \bigl( \| \eta(t) \|^2_4 +  \| \rho_0  J^{-2} (t)\|^2_4\bigr) \le 
M_0 + \delta {\sup_{t\in[0,T]}} E(t) + C \, T\, P({\sup_{t\in[0,T]}} E(t)) \,.
\end{equation} 
\subsection{Estimates for $ \operatorname{curl} _\eta v$} 
 As a result of the inequality (\ref{ss777_2d}) and the
identity $ \operatorname{curl}_\eta v = \operatorname{curl} u_0 J^{-1} $, by readjusting the constants if necessary we
have that
$$
\sup_{t \in [0,T]} \bigl( \| \operatorname{curl} _ \eta v(t)\|^2_3 + \| \rho_0 \bar \p^4 \operatorname{curl} _\eta v(t)\|^2_0 \bigr) \le M_0 + \delta \sup_{t \in [0,T]} E(t) + 
C\,T\ P( \sup_{t \in [0,T]} E(t)) \,.
$$
\begin{corollary} \label{cor_curlv_3d}
$$
\sup_{t \in [0,T]} \bigl( \| \operatorname{curl} _ \eta v(t)\|^2_3 + \| \rho_0 \bar \p^4 \operatorname{curl} _\eta v(t)\|^2_0 \bigr) \le M_0 + \delta \sup_{t \in [0,T]} E(t) + 
C\,T\ P( \sup_{t \in [0,T]} E(t)) \,.
$$\end{corollary}
\begin{proof}
Letting $D^3$ act on the identity (\ref{curlv_3d}) for $ \operatorname{curl} _\eta v$, we see that the highest-order term scales like
$$
D^3 \operatorname{curl} u_0 + \int_0^t D^4 v \, Dv\, A \, A dt' \,.
$$
We integrate by parts to see that the highest-order contribution to $ D^3 \operatorname{curl} _\eta v(t)$ can be written as
$$
D^3 \operatorname{curl} u_0 - \int_0^t D^4 \eta \, [Dv\, A \, A]_t dt'  + D^4\eta(t)\, Dv(t)\, A(t)\, A(t) \,,
$$
which, according to (\ref{ss777_2d}), has $L^2(\Omega)$-norm bounded by $$ M_0(\delta) + \delta  {\sup_{t\in[0,T]}} E(t) +
C \, T\, P({\sup_{t\in[0,T]}} E(t))\,,$$ after readjusting the constants; thus, the inequality for the $H^3(\Omega)$-norm of $ \operatorname{curl} _\eta
v(t)$ is proved

The same type of analysis works for the weighted estimate.
After integration by parts in time,  the highest-order term in the expression for $\rho_0 \bar \p^4 \operatorname{curl} _\eta v(t)$ scales like
$$
\rho_0 \bar \p^4 \operatorname{curl} u_0 - \int_0^t \rho_0 \bar \p^4 D \eta \, [Dv\, A \, A]_t dt'  + \rho_0 \bar \p^4D\eta(t)\, Dv(t)\, A(t)\, A(t) \,.
$$
Hence, the inequality (\ref{energy1_2d}) shows that the weighted estimate holds as well.
\end{proof}

\subsection{The a priori bound}
Summing the inequalities provided by our energy estimates, the additional elliptic estimates, and the estimates for 
$ \operatorname{curl} _\eta v$ shows that
$$
\sup_{t \in [0,T]} E(t) \le M_0 + C\, T\, P( \sup_{t \in [0,T]} E(t)) \,.
$$
According to our polynomial-type inequality given in Section \ref{subsec_poly}, by taking $T>0$ sufficiently small,
we have the a priori bound
$$
\sup_{t \in [0,T]} E(t) \le 2M_0\,.
$$

\vspace{.1 in}

\noindent
{\bf Acknowledgments.}
SS was supported by the National Science Foundation under
grant DMS-0313370.  HL was supported by the National Science Foundation under
grant DMS-0801120

\end{document}